\documentclass[12pt,leqno,twoside]{amsart}
\usepackage{amsmath,amscd,amssymb,amsfonts,latexsym,mathrsfs,bm,extarrows,graphicx,stackrel}
\usepackage[all,cmtip]{xy}
\usepackage{ulem}
\usepackage{tikz}

\usepackage{amssymb}
\usepackage{latexsym}
\usepackage{graphicx}
\usepackage{epsfig}
\usepackage{srcltx}
\usepackage[colorlinks=true, pdfstartview=FitV, linkcolor=blue, citecolor=blue, urlcolor=blue]{hyperref}
\usepackage{subcaption}

\newtheorem{theorem}{Theorem}[section]
\newtheorem{corollary}[theorem]{Corollary}
\newtheorem{lemma}[theorem]{Lemma}
\newtheorem{proposition}[theorem]{Proposition}

\theoremstyle{remark}
\newtheorem{remark}[theorem]{\sc Remark}
\theoremstyle{remark}

\theoremstyle{definition}

\theoremstyle{remark}
\newtheorem{example}[theorem]{\sc Example}

\theoremstyle{remark}

\theoremstyle{remark}

\numberwithin{equation}{section}  % numbers as (1.2) instead of (1) etc
\def\be{\begin{equation}}
\def\ee{\end{equation}}
\def\bt{\begin{theorem}}
\def\et{\end{theorem}}
\newcommand{\lra}{\longrightarrow}
\def\bc{\begin{corollary}}
\def\ec{\end{corollary}}
\def\br{\begin{remark}}
\def\er{\end{remark}}
\def\bex{\begin{example}}
\def\eex{\end{example}}
\def\bp{\begin{proposition}}
\def\ep{\end{proposition}}

\usepackage{srcltx}
\usepackage[colorlinks=true, pdfstartview=FitV, linkcolor=blue, citecolor=blue, urlcolor=blue]{hyperref}
\usepackage{subcaption} 
\usepackage{amsbsy}
\usepackage{amsthm,alltt,dsfont}

\usepackage[margin=1.2in]{geometry}

\usepackage[utf8]{inputenc}
\usepackage{amsthm,soul}
\usepackage{enumitem}
\usepackage{setspace,kantlipsum}
\usepackage{tikz-cd}
\usepackage{textcomp}

\usepackage{colonequals}
\usepackage{mathtools,comment}

\usepackage{graphicx,bm}
\usepackage{mathptmx}

\newcommand{\C}{\mathbb{C}}

\newcommand{\Q}{\mathbb{Q}}

\newcommand{\cD}{\mathcal{D}}

\newcommand{\cF}{\mathcal{F}}
\newcommand{\cG}{\mathcal{G}}
\newcommand{\cH}{\mathcal{H}}

\newcommand{\cL}{\mathcal{L}}

\newcommand{\cP}{\mathcal{P}}

\newcommand{\cS}{\mathcal{S}}

\newcommand{\cW}{\mathcal{W}}

\newcommand{\bC}{\mathbb{C}}

\newcommand{\bH}{\mathbb{H}}

\newcommand{\bQ}{\mathbb{Q}}

\newcommand{\bZ}{\mathbb{Z}}

\newcommand{\of}{\overline{f}}
\newcommand{\oX}{\overline{X}}
\newcommand{\oF}{\overline{F}}
\newcommand{\ok}{\overline{k}}
\newcommand{\oi}{\overline{i}}
\newcommand{\oY}{\overline{Y}}
\newcommand{\oE}{\overline{E}}
\newcommand{\og}{\overline{g}}
\newcommand{\oG}{\overline{G}}

\newcommand{\wti}{\widetilde}

\newcommand{\Sing}{\mathrm{Sing}}

\begin{document}

\title[Fibers of polynomial maps]
{On the topology of fibers of complex polynomial maps}

\author{Lauren\c{t}iu Maxim}
\address{L. Maxim : Department of Mathematics, University of Wisconsin-Madison, 480 Lincoln Drive, Madison WI 53706-1388, USA, \newline {\text and} \newline Institute of Mathematics of the Romanian Academy, P.O. Box 1-764, 70700 Bucharest, ROMANIA}
\email{maxim@math.wisc.edu}

\author{John Messina}
\address{J. Messina : Department of Mathematics, University of Wisconsin-Madison, 480 Lincoln Drive, Madison WI 53706-1388, USA}
\email{jamessina@wisc.edu}

\dedicatory{Dedicated to the memory of our friend and colleague Mihai Tib\u{a}r.}

\begin{abstract}
%The aim of this paper is twofold. On the one hand, we survey known results on the cohomology of fibers of complex polynomial maps, with particular emphasis on vanishing ranges. On the other hand, we establish new results on the vanishing cohomology of a polynomial at a bifurcation value and derive upper bounds for the first possibly nonvanishing Betti number of a generic, respectively atypical, fiber in terms of local singularity invariants. Our results generalize several theorems of Tibăr, Siersma, Dimca, and others, previously proved for isolated singularities, including at infinity.

We survey known results on the cohomology of fibers of complex polynomial maps, with particular emphasis on vanishing ranges, and establish new results on the vanishing cohomology of a polynomial at a bifurcation value. We also derive sharp upper bounds for the first possibly nonvanishing Betti number of general and atypical fibers in terms of local singularity invariants. These results extend several theorems of Tib\u{a}r, Siersma, Dimca, and others from the case of isolated singularities, including singularities at infinity, to polynomial maps with arbitrary singularities.
\end{abstract}

\date{\today}



\maketitle

\tableofcontents

\section{Introduction}
\subsection{Setup} Let $f\colon \mathbb{C}^{n+1}\rightarrow \mathbb{C}$ be a complex polynomial map. Then there is a finite set $B \subset \mathbb{C}$ such that 
\[f|_{\mathbb{C}^{n+1}\setminus f^{-1}(B)}: \mathbb{C}^{n+1}\setminus f^{-1}(B)\rightarrow \mathbb{C}\setminus B\]
is a topologically locally trivial fibration. The set $B$ consists of the collection of ``bifurcation values'' of $f$. It contains the critical values of $f$ and some other points which can be interpreted as the images of the ``critical points of $f$ at infinity''. 

An important problem in singularity theory is to understand the topology of fibers of the polynomial $f$, including both general and atypical fibers (the latter corresponding to bifurcation values). This is a problem with a long history, e.g., see \cite{Bro, NN, Hamm, Di, Sab, NZ, DS, ST, Pa, ALM, Di0, Bre, T0, T1, T2, Ta}, etc. 

For $c\notin B,$ the above fibration induced by $f$  ensures that the topology of the fiber \[F_c=f^{-1}(c)\] is independent of $c$. We will then occasionally use $F$ to refer to a general fiber $F_c$ for $c\notin B$. Note that the fibers of $f$ are $n$-dimensional complex affine varieties, so their cohomology vanishes in degrees $k>n$. 

As shown in \cite{Bro, NN, ST} (see also the treatments in \cite{Di, T2}), the topology of the general fiber $F$ of $f$ is determined by the vanishing topology at the bifurcation values. More precisely, if $D_b \subset \bC$ is a sufficiently small disc around $b \in B$ and \[T(F_b):=f^{-1}(D_b)\] is the corresponding ``tube'' around the atypical fiber $F_b=f^{-1}(b)$, the $k$-th ``vanishing cohomology'' at $b$ (or ``vanishing cocycles'') is defined as
\[ V^k(b):=H^{k+1}(T(F_b), F;\bQ),\]
for $F$ a general fiber of $f$ contained in $T(F_b)$. (A similar notion was considered and studied in \cite{MPT1} in the projective context.) Then an easy deformation retraction argument yields the following isomorphism for any $k\geq 0$:
\begin{equation}\label{iso-van}
\widetilde{H}^k(F;\bQ)\cong \bigoplus_{b\in B} V^k(b).
\end{equation}
Moreover, the cohomology long exact sequence for the triple $(\bC^{n+1}, T(F_b), F)$ yields short exact sequences
\be\label{ses}
0\lra \widetilde{H}^k(T(F_b);\bQ) \lra \widetilde{H}^k(F;\bQ) \lra V^k(b) \lra 0,
\ee
for each $b \in B$.

In order to study the cohomology of $F$ and $F_b$ ($b \in B$), it is convenient to consider a proper extension $\overline{f}$ of $f$. For convenience and concreteness, we will work with the closure of the graph of $f$ (but see also \cite{Bre, Di} for more general situations). In more detail, if $f\colon \mathbb{C}^{n+1}\rightarrow \mathbb{C}$ is a degree $d$ polynomial, with homogenization 
$f^h(x_0,x)$ by the new variable $x_0$, the closure in $\bC P^{n+1} \times \bC$ of the graph $X\cong \C^{n+1}$ of $f$ is the hypersurface
\[
\overline{X}:=\{\left([x_0:x],t\right) \in  \bC P^{n+1} \times \bC \mid f^h(x_0,x)-tx_0^d=0  \}.
\]
It comes endowed with an open (affine) inclusion 
$j: X\cong \mathbb{C}^{n+1}\hookrightarrow \overline{X}$ given by 
 \[(x_1,\ldots, x_{n+1})\mapsto ([1: x_1: \ldots : x_{n+1}], f(x_1,\ldots x_{n+1})),\]
and the projection $\overline{f}: \overline{X}\rightarrow \mathbb{C}$ onto the second coordinate. Then $\overline{f}$ is a proper extension of $f$, whose fibers are isomorphic to the projective closures of the fibers of $f$. Let $H_\infty=\{x_0=0\}$ denote the hyperplane at infinity in $\bC P^{n+1}$, and 
\[ \overline{X}_\infty:=\overline{X} \cap (H_\infty \times \bC)\]
be the part ``at infinity'' of $\overline{X}$. Then the singularities of $f$ can be identified with the singularities of $\overline{f}$ on $\overline{X} \setminus \overline{X}_\infty$. Let us also note that $B$ can be chosen as the minimal subset of $\bC$ such that $\overline{f}$ is a stratified submersion over $\bC \setminus B$.

Throughout this paper we will assume that $\dim \Sing(f)<n$, which in turn guarantees that the general fiber of $f$ is connected. (Such polynomials are called ``primitive'', e.g., see \cite[Prop.6.3.5]{Di}.)

\subsection{Statement of results}
We now give a brief account of the results discussed in this paper. For a gentle introduction to background material on constructible and perverse sheaves, as well as on nearby and vanishing cycles, see e.g., \cite{Di, Max, MS}.

\medskip

In Section \ref{vc}, we undertake the study of vanishing cohomology at a bifurcation value $b \in B$. The main tool in this analysis is the vanishing cycle complex $\varphi_{\overline{f}-b}Rj_*\bQ_X$ associated to the function $\of -b$, with  
$j\colon X\cong \mathbb{C}^{n+1}\hookrightarrow \overline{X}$ the open affine inclusion introduced above. Indeed, as shown, e.g., in \cite{Bre}, one gets a vector space isomorphism (see Proposition \ref{p21})
\be\label{eqi1} V^k(b)\cong \bH^k(\overline{F}_b; \varphi_{\overline{f}-b}Rj_*\bQ_X) \ee
for any non-negative integer $k$. Here \[\oF_b=\{\of=b\}\] is the fiber of $\of$ over $b$, and also the projective closure of $F_b$.

Let $\Sigma_b \subseteq \oF_b$ denote the support of $\varphi_{\overline{f}-b}Rj_*\bQ_X$. Using standard properties of perverse sheaves,  it follows immediately from \eqref{eqi1} that
\be\label{eqi2} V^k(b) \cong 0, \ {\rm for \ all} \ k \notin [n-s_b, n],   \ee
where $s_b=\dim \Sigma_b$.

Our Theorem \ref{t14a} investigates the vanishing cohomology groups $V^k(b)$ in the (possibly) non-vanishing range $k \in [n-s_b, n]$. Since the case $s_b=0$ is pretty well understood (e.g., see Corollary \ref{iso} and the references therein), in what follows we shall assume that $s_b>0$. While the precise statement of Theorem \ref{t14a} is likely too technical for inclusion in this introduction, one of its consequences is that, if $s_b \geq 2$ and $k<n-1$, the group $V^k(b)$ can be computed from the restriction of the above-mentioned vanishing cycle complex to the singular strata of $f$ contained in $\Sigma_b$ and  of dimension at least $n-k-1$ (see Corollary \ref{c2.2}). This result can be viewed as a global counterpart to the Milnor fiber cohomology considerations from  \cite{MPT2}, and it reduces the computation of vanishing cohomology to a hypercohomology spectral sequence. Although this computation is generally tedious, it becomes more explicit in the case of the lowest (possibly) nontrivial vanishing cohomology group $V^{n-s_b}(b)$; see  Theorem \ref{c14a}. More precisely, if $s_b \geq 1$, our Theorem \ref{c14a} yields an upper bound on the dimension of $V^{n-s_b}(b)$, namely we show that
\be\label{upbi}
\dim V^{n-s_b}(b) \leq  \sum_i  \mu^\pitchfork_{s_b,i},
 \ee
 where the summation is over the collection $\{\Sigma_b^{s_b,i}\}_i$ of $s_b$-dimensional connected strata of $\Sigma_b$ (including at infinity), and 
\begin{center} $\mu^\pitchfork_{s_b,i}:= \dim \widetilde{H}^{n-s_b}(MF^\pitchfork_{s_b,i};\bQ)$\end{center} is the {\it transversal Milnor number} of $\of-b$ along $\Sigma_b^{s_b,i}$,
with $MF^\pitchfork_{s_b,i}$ denoting the transversal Milnor fiber of $\of-b$ at a point in the stratum $\Sigma_b^{s_b,i}$.
%injective homomorphism \be\label{upbi} V^{n-s_b}(b) \hookrightarrow \bigoplus_i \widetilde{H}^{n-s_b}(MF^\pitchfork_{s_b,i};\bQ)^{A_b^i}, \eewhere the summation is over the collection $\{\Sigma_b^{s_b,i}\}_i$ of $s_b$-dimensional connected strata of $\Sigma_b$,  $MF^\pitchfork_{s_b,i}$ is the transversal Milnor fiber of $\of -b$ to the $s_b$-dimensional stratum $\Sigma_b^{s_b,i}$, and $A_b^i$ denotes the action of $\pi_1(\Sigma_b^{s_b,i})$  on $\widetilde{H}^{n-s_b}(MF^\pitchfork_{s_b,i};\bQ)$, with invariant subspaces appearing on the right hand side of \eqref{upbi}. 
If $s_b \geq 1$, upper bounds similar to \eqref{upbi} can be deduced from Theorem \ref{t14a} for all vanishing cohomology groups $V^k(b)$ in the range $n-s_b \leq k \leq n-1$, though not as explicit as in the case $k=n-s_b$.

It is important to note that the computation of the cohomology stalks of the vanishing cycle complex $\varphi_{\of-b}Rj_*\bQ_X$ at points in $\Sigma_b^\infty:=\Sigma_b \cap \oX_\infty$ in terms of the reduced Milnor fiber cohomology of $\of -b$ at such points (cf. \eqref{mfi} in Lemma \ref{identif}) is the main reason for working with the compactification $\oX$ of $X$ given by the graph closure of $f$.

 \medskip
 
 In Section \ref{genf}, we assemble the information on vanishing cohomology obtained in Section \ref{vc} to study the cohomology of the general fiber $F$ of $f$. With the above notation, set \[ s:=\max_{b\in B} s_b.\] We assume $s \geq 0$, and say that ``$f$ has $s$-dimensional singularities, including at infinity''.  We also define
\[B_s:=\{b \in B \mid s_b=s\}.\]
In view of the isomorphism \eqref{iso-van}, the vanishing in \eqref{eqi2} yields a well-known vanishing result (cf. \cite{NN,ST,Bre,Di}) for the cohomology of the general fiber $F$ of $f$, together with a description of its lowest (possibly) non-vanishing cohomology group. More precisely, one has (cf. Corollary \ref{c26})
\be\label{eqi3} \widetilde{H}^k(F;\bQ)\cong 0,  \ \ {\rm for} \ k \notin [n-s,n], \ee together with
\be\label{nmsgeni}
\widetilde{H}^{n-s}(F;\bQ)\cong \bigoplus_{b \in B_s} V^{n-s}(b).
\ee
In particular, by combining \eqref{nmsgeni} with \eqref{upbi}, we get an upper bound on the lowest (possibly) non-vanishing Betti number $b_{n-s}(F)=\dim \wti{H}^{n-s}(F;\bQ)$ of the general fiber of $f$. 
As the case $s=0$ was already considered in, e.g., \cite{ST,Hamm,Di}, here we assume that there exists a bifurcation value $b \in B$ with $s_b>0$, hence also $s>0$. We then get in Corollary \ref{nmsb} that
 \be\label{nmsgbi}
b_{n-s}(F) \leq  \sum_{b \in B_s} \sum_i  \mu^\pitchfork_{b,s,i},
 \ee
with $\mu^\pitchfork_{b,s,i}$ denoting as before the transversal Milnor  number of $\of-b$ at an $s$-dimensional stratum in the atypical fiber $\oF_b$, for some $b \in B_s$, and $i$ indexing the connected $s$-dimensional strata in $\Sigma_b$.

%the summation is over (the connected components of) all such strata.

%an injective homomorphism  \be\label{nmsgen2i} \widetilde{H}^{n-s}(F;\bQ)\hookrightarrow \bigoplus_{b \in B_s}  \bigoplus_i \widetilde{H}^{n-s}(MF^\pitchfork_{s,i};\bQ)^{A_b^i},\ee and a corresponding upper bound for the $(n-s)$-th Betti number of $F$, with $MF^\pitchfork_{s,i}$  denoting as before the transversal Milnor fiber of $\of-b$ at an $s$-dimensional stratum in the atypical fiber $\oF_b$, for some $b \in B_s$. It is worth noting that if $s\geq 2$ and, for every $b \in B_s$, the set $\Sigma_b$ contains no strata of dimension $s-1$, then \eqref{nmsgen2i} is an isomorphism.

\medskip

In Section \ref{atyp}, we describe vanishing results for the cohomology of an atypical fiber $F_b$ of $f$, for $b \in B$, as well as an upper bound on the least (possibly) non-vanishing Betti number of $F_b$, with no restriction on the dimension of the singularities. Some of these results already appear in \cite[Sections 6.2, 6.3]{Di} in the case of isolated singularities, including at infinity. 

Assume that the complex polynomial $f\colon \bC^{n+1} \to \bC$ has $s$-dimensional singularities, including at infinity, and let $b \in B$ be a bifurcation value of $f$, with corresponding atypical fiber $F_b=f^{-1}(b)$. Then one has (cf. Theorem \ref{tat}):
\be\label{eqi4}
\wti{H}^k(F_b;\bQ)=0, \ \ {\rm for \ all} \ \ k < n-s-1.
\ee
Moreover, if we assume that 
\[\dim(\Sigma_b^\infty) <s,\] 
e.g., if $s_b=\dim \Sigma_b <s$, then 
\[  \wti{H}^{n-s-1}(F_b;\bQ)=0.\]

The vanishing in \eqref{eqi4} has already been established in a more general setting in \cite{Bre}. However, by using the compactification $\oX$ of $\bC^{n+1}$ given by the closure of the graph $X$ of $f$, a careful analysis of the proof of \eqref{eqi4}, combined with a hypercohomology spectral sequence, yields an upper bound on the least (possibly) non-vanishing Betti number of an atypical fiber $F_b$, with $\dim(\Sigma_b^\infty) =s$, namely (see Theorem \ref{topbb} for a sharper statement):
\be\label{bboundi}
\dim  \wti{H}^{n-s-1}(F_b;\bQ) \leq \sum_{S \subset \Sigma_b^\infty \atop \dim S=s} \mu^\pitchfork_S,
\ee
where the sum runs over all $s$-dimensional Whitney strata contained in $\Sigma_b^\infty$, and $\mu^\pitchfork_S:=\dim \wti{H}^{n-s}(MF^\pitchfork_{S};\bQ)$ is the corresponding transversal Milnor number of $\of-b$ along such an $s$-dimensional stratum $S \subset \Sigma_b^\infty$.

%\be\label{bboundi}\dim  \wti{H}^{n-s-1}(F_b;\bQ) \leq \sum_{S \subset \Sigma_b^\infty \atop \dim S=s} \dim \ \ker \left(M_b - id:  \wti{H}^{n-s}(MF^\pitchfork_{S};\bQ)^{A_S} \to    \wti{H}^{n-s}(MF^\pitchfork_{S};\bQ)^{A_S} \right),\ee Here, for an $s$-dimensional connected stratum $S$ in $\Sigma_b^{\infty}$ (and hence also a top stratum in $\Sigma_b$), $MF^\pitchfork_{S}$ denotes the (transversal) Milnor fiber of $\of-b$ at some point in $S$, $A_S$ is the action of $\pi_1(S)$ on $\wti{H}^{n-s}(MF^\pitchfork_{S};\bQ)$, and $M_b$ is the automorphism induced from the monodromy of the vanishing cycle complex $ \varphi_{\overline{f}-b}Rj_*\bQ_X$. Moreover, the sum on the right-hand side of the inequality \eqref{bboundi} runs over all $s$-dimensional Whitney strata contained in $\Sigma_b^\infty$. 

%\medskip

%In Section \ref{slice}, we recover several of the above-mentioned results by reduction to the case of isolated singularities via generic slicing (see Proposition \ref{prop51} and Theorem \ref{gensl} for details).

\medskip

In Section \ref{top} we discuss a sharp upper bound for the top Betti number of the general fiber of a complex polynomial. More precisely, in Proposition \ref{pr61} we show that if $F$ is the general fiber of a degree $d$ complex polynomial $f\colon \C^{n+1} \to \C$, then (cf. also \cite{ST1,ST2,T2})
\be\label{tbi}
b_n(F)\leq (d-1)^{n+1}.
\ee
Moreover, \eqref{tbi} becomes an equality if $f$ is of $\cG$-type in the sense of \cite{T2}, e.g., a Fermat polynomial. This result is an immediate consequence of the  ``lower semicontinuity property'' for the top Betti number of general fibers, a homological version of which was already proved in \cite{ST1,T2} by homotopy-theoretic methods. In Theorem \ref{lsc} we give  a sheaf theoretic proof of this lower semicontinuity property, which has the advantage that it is also amenable to Hodge-theoretic considerations. 

%Specifically, we prove the following.
%\begin{theorem} Let $P(x,s):\C^{n+1} \times \C \to \C$ be a one-parameter family of degree $d$ polynomials, and denote by $F^s$ the general fiber of $f_s:=P(-,s)$, $s \in \C$. Then there is an injective homomorphism \be\label{injli} H^n_c(F^0;\Q)  \rightarrowtail H^n_c(F^s;\Q) \ee for $s\neq 0$ close enough to $0$. In particular, \begin{center} $b_n(F^s) \geq b_n(F^0)$, for $s \neq 0$ close enough to $0$.\end{center} \end{theorem}

\medskip

Several concrete examples are worked out in Section \ref{ex}. In particular, Example \ref{ex61} shows that the bounds obtained in \eqref{nmsgbi} (see Corollary \ref{nmsgb}) and \eqref{bboundi} (see Theorem \ref{topbb}) are {\it sharp}.

\medskip

We conclude this introduction with several remarks. First, the polynomial $f$ may be replaced by a non-constant morphism $f:X \to C$ from a smooth complex algebraic variety $X$ to a curve. One may then work with a proper extension $\of$ of $f$ defined on a partial compactification $\oX$ of $X$, assuming that the inclusion $X\hookrightarrow \oX$ is affine. The vanishing results remain valid in this broader setting (see, for example, \cite{Bre}), and it is also possible to formulate abstract Betti-number bounds, cf. \cite{Di} for the isolated singularities case.
Second, all complexes of sheaves considered in this paper underlie complexes of mixed Hodge modules, and therefore their hypercohomology groups carry natural mixed Hodge structures. The techniques developed here can be adapted to derive analogous upper bounds for the corresponding mixed Hodge numbers. These refinements will be explored in future work.

\subsection{Acknowledgments} 
The authors would like to express their heartfelt thanks to Mihai Tib\u{a}r. His book on the subject has been an essential guide, and we are especially grateful for his patience in explaining the concepts and subtleties of his work. His generosity and willingness to share his expertise have been truly inspiring and have greatly shaped our understanding of this area of research.

L. Maxim acknowledges support from the Simons Foundation and from the project ``Singularities and Applications'' - CF 132/31.07.2023 funded by the European Union - NextGenerationEU - through Romania's National Recovery and Resilience Plan. 

%%%%%%%%%%%%%%%%%%%%%%%

\section{Vanishing cohomology at a bifurcation value}\label{vc}
In this section we study the vanishing cohomology groups associated to a fixed bifurcation value $b\in B$ of $f$. In view of \eqref{iso-van}, this will in turn provide useful information about the cohomology of the general fiber $F$ of $f$ (see the subsequent Section \ref{genf}). We propose here a unifying approach, from which we recover and extend some of the results from \cite{Bre,Di}.

We begin with the following motivating example.
\bex[Weighted homogeneous polynomial]\label{whp}
Assume $f\colon \bC^{n+1} \to \bC$ is a weighted homogeneous polynomial of degree $d$. % with respect to the positive weights $w_i$ for each variable $x_i$, $i=1,\ldots, n+1$. 
Then $$f\colon \bC^{n+1} \setminus f^{-1}(0) \to \bC^*$$ is a locally trivial fibration, referred to as the global Milnor fibration of $f$. In particular, $0 \in \bC$ may be the only bifurcation value of $f$, so $B\subseteq\{0\}$. (Of course, $B=\{0\}$ if $f$ is homogeneous of degree $\geq 2$. On the other hand, there exist weighted homogeneous polynomials like $f(x_1,x_2)=x_1+x_2^2$  which are fibrations, so in this case $B=\emptyset$.) The general fiber of $f$ (also called in this case the Milnor fiber of $f$) is homotopy equivalent to the Milnor fiber associated to the germ of $f$ at the origin of $\bC^{n+1}$. Note that \eqref{iso-van} becomes in this case just $\widetilde{H}^k(F;\bQ)\cong V^k(0) \cong H^{k+1}(\bC^{n+1}, F; \bQ)$. $\hfill\qed$
\eex

The non-vanishing cohomology of the Milnor fiber of a complex hypersurface singularity germ has been studied in detail, for example in \cite{MPT2}. In particular, the approach developed in \cite{MPT2} motivates the methods used in this section to investigate the vanishing cohomology.

\smallskip

Let $f\colon \bC^{n+1}\rightarrow \mathbb{C}$ be a complex polynomial map as above, and let $\overline{f}\colon \overline{X} \to \bC$ denote the proper extension constructed via the closure $\oX$ of the graph $X$ of $f$. 
We begin with the following important result (see, for example, \cite[Prop.2.2]{Bre}), and include its proof for the reader’s convenience.
\begin{proposition}\label{p21}
For any bifurcation value $b\in B$ of $f$ and non-negative integer $k$, one has the isomorphisms:
\begin{equation}\label{vanb}
V^k(b)\cong \bH^k(\overline{F}_b; \varphi_{\overline{f}-b}Rj_*\bQ_X) \cong H^k(\varphi_{t-b}Rf_*\bQ_X) \cong \varphi_{t-b}(R^kf_*\bQ_X),
\end{equation}
where $\overline{F}_b=\overline{f}^{-1}(b)$, $j\colon X\cong \C^{n+1}\hookrightarrow \overline{X}$ is the affine inclusion, and $t$ is the coordinate function on $\bC$.
\end{proposition}

\begin{proof} The second isomorphism in \eqref{vanb} follows by proper base change for vanishing cycles (e.g., see \cite[Prop.4.2.12]{Di}), using the fact that $\of$ is proper and $f=\of \circ j$. For the identification between the first and last groups in \eqref{vanb}, see, e.g., \cite[Prop.6.3.6]{Di}. The isomorphism between the first and third groups in \eqref{vanb} follows by using the comparison triangle
\[ (Rf_*\bQ_X)_b \overset{comp}{\lra} \psi_{t-b}Rf_*\bQ_X \lra \varphi_{t-b}Rf_*\bQ_X \overset{[1]}{\lra} \]
and the associated long exact cohomology sequence, after noting that:
\begin{itemize}
\item[(i)] $H^k((Rf_*\bQ_X)_b)\cong H^k(D_b; Rf_*\bQ_X)\cong H^k(T(F_b);\bQ)$.
\item[(ii)] For $c\in D_b$, $c\neq b$,  the Cartesian diagram
\[
\xymatrix{
F_c=f^{-1}(c) \ar[d]_{j_c}\ar[r]^{k_c}&X\cong\C^{n+1} \ar[d]^{j}\\
\oF_c=\of^{-1}(c) \ar[r]_{{\overline k}_c}& \oX,
}
\]
yields the following isomorphisms
\[ \begin{split} 
H^k( \psi_{t-b}Rf_*\bQ_X ) &\cong \bH^k(\overline{F}_b; \psi_{\overline{f}-b}Rj_*\bQ_X) 
\overset{(a)}{\cong} \bH^k(\overline{F}_c; {\overline k}_c^*Rj_*\bQ_X) \overset{(b)}{\cong}  \bH^k(\overline{F}_c; R(j_c)_*{k}_c^*\bQ_X) \\ 
&\cong  H^k(F_c; \bQ), \\ 
\end{split} \]
where the first isomorphism follows by proper base change  for nearby cycles (see \cite[Prop.4.2.12]{Di}), (a) follows, e.g., from \cite[Ex.10.4.20]{MS}, 
and (b) uses the base change isomorphism (e.g., see  \cite[Prop.4.3.1, Rem. 4.3.6]{Sch}).
\item[(iii)] The comparison map $comp$ induces in cohomology the same morphism as the one induced by the embedding $F_c \hookrightarrow T(F_b)$.
\end{itemize}
\end{proof}

As Proposition \ref{p21} shows, in order to study the vanishing cohomology of $f$ at $b \in B$ it suffices to understand homological properties of the complex $\varphi_{\overline{f}-b}Rj_*\bQ_X$. Let us choose a Whitney stratification $\cW$ of $\oX$ with respect to which $Rj_*\bQ_X$ is constructible. Then 
\be\label{sup} {\rm Supp} \ \varphi_{\overline{f}-b}Rj_*\bQ_X \subseteq \oF_b \cap \Sing_{\cW}(\of),\ee
where $\Sing_{\cW}(\of)=\bigcup_{S \in \cW} \Sing(\of\vert_S)$ is the stratified singular locus of $\of$. 
Let 
\[\Sigma_b:={\rm Supp} \ \varphi_{\overline{f}-b}Rj_*\bQ_X,\] and denote $s_b:=\dim \Sigma_b$. One has the following vanishing result (e.g., see \cite[Prop.3.1]{Bre}).
\begin{corollary}\label{cor32}
For any $b \in B$, $V^k(b)\cong 0$ for all $k \notin [n-s_b, n]$. 
\end{corollary}
\begin{proof}
Since $j$ is a quasi-finite affine map and $\bQ_X[n+1]\in Perv(X)$ is a perverse sheaf, we have that 
$Rj_*\bQ_X[n+1] \in Perv(\oX)$. Since the perverse vanishing cycle functor $^p\varphi_{\of}:=\varphi_{\of}[-1]$ preserves perverse sheaves, it follows that $\varphi_{\overline{f}-b}Rj_*\bQ_X[n] \in Perv(\oX)$. Furthermore, since $\varphi_{\overline{f}-b}Rj_*\bQ_X[n]$ is supported on $\Sigma_b$, its restriction 
\be\label{pp} \cP_b:=\varphi_{\overline{f}-b}Rj_*\bQ_X[n] \vert_{\Sigma_b} \ee
is a perverse sheaf on $\Sigma_b$. Formula \eqref{vanb} and the support condition for perverse sheaves then yield that
\begin{center} $V^k(b) \cong \bH^{k-n}(\Sigma_b; \cP_b) \cong 0$ for $k-n \notin [-s_b,s_b]$.
\end{center}
Finally,  it follows from \eqref{iso-van} that $V^k(b)=0$ for all $k>n$.
\end{proof}

As a special case, we get the following (see  \cite{ST}, \cite{Hamm}, and also \cite[Prop.6.2.19(i)]{Di}).
\bc\label{iso}
If, in the above notations, we assume moreover that $s_b=0$, then $V^k(b)\cong 0$ for all $k \neq n$, and
\be\label{e23}
V^n(b)\cong \bigoplus_{p\in \Sigma_b} \cH^n(\varphi_{\of-b}Rj_*\bQ_X)_p. 
\ee
\hfill\qed
\ec

For $b \in B$ a bifurcation value of $f$ as above, let 
\begin{center} $\Sigma^a_b:=\Sigma_b \cap X=\Sing(f) \cap F_b$ and $\Sigma_b^\infty:=\Sigma_b \cap \oX_\infty$,\end{center} so that $\Sigma_b=\Sigma^a_b \cup \Sigma_b^\infty$. For $p\in \Sigma^a_b$ and $k \in \bZ$, one has the isomorphism
\be\label{mfa} \cH^k(\varphi_{\of-b}Rj_*\bQ_X)_p \cong \cH^k(\varphi_{f-b}\bQ_X)_p \cong \widetilde{H}^k(MF_{f-b,p};\bQ),\ee
with $MF_{f-b,p}$ denoting the Milnor fiber of $f-b$ at the singular point $p$ (which in this case can be identified with the Milnor fiber of $\of-b$ at $p$). For $p\in \Sigma^\infty_b$, one can make use of the special geometry of the compactification $\oX$ to show the following identification:

\begin{lemma}\label{identif}
Assume that $s_b=0$, and let $p\in \Sigma^\infty_b=\Sigma_b \cap \oX_\infty$. Then
\be\label{mfi} \left(\cH^n(\varphi_{\of-b}Rj_*\bQ_X)_p\right)^\vee \cong \cH^n(\varphi_{\of-b}\bQ_{\oX})_p \cong \widetilde{H}^n(MF_{\of-b,p};\bQ),\ee
with $MF_{\of-b,p}$ the Milnor fiber of $\of-b$ at the point $p\in \Sigma^\infty_b$, and $(-)^\vee$ denoting the dual of a rational vector space. In particular,
\be\label{mfi2}
\dim \left( \cH^n(\varphi_{\of-b}Rj_*\bQ_X)_p \right) =\dim \left( \cH^n(\varphi_{\of-b}\bQ_{\oX})_p \right)= \dim \widetilde{H}^n(MF_{\of-b,p};\bQ).
\ee
\end{lemma}

\begin{proof} In order to simplify the notations, we can assume without any loss of generality that $b=0$, and write $\Sigma^\infty$ in place of $\Sigma^\infty_b$. Furthermore, by restricting $f$ over a sufficiently small disc at the origin, we may further assume that $b=0$ is the only bifurcation value.

Let $j:X  \hookrightarrow \oX$ and $i: \oX_\infty  \hookrightarrow \oX$ be the inclusion maps. Consider the distinguished triangle
\[ j_!\Q_X \lra \Q_{\oX} \lra i_* \Q_{\oX_\infty} \overset{[1]}{\lra} \]
and apply the functor $\varphi_{\of}$ to it to get the triangle
\begin{equation}\label{tr5}
\varphi_{\of} j_!\Q_X \lra \varphi_{\of} \Q_{\oX} \lra \varphi_{\of} i_* \Q_{\oX_\infty} \overset{[1]}{\lra}
\end{equation}
Using the base change property for vanishing cycles together with the product structure of $\oX_\infty$, one can show as in \cite[Ex.6.2.20, p.190]{Di} that 
\begin{equation}\label{tr6}
\varphi_{\of} i_* \Q_{\oX_\infty} \cong 0.
\end{equation}
It then follows from \eqref{tr5}  and \eqref{tr6} that
\begin{equation}\label{tr7}
\varphi_{\of} j_!\Q_X \cong \varphi_{\of} \Q_{\oX}.
\end{equation}
Next note that since $X$ is smooth, $\Q_X[n+1]$ is Verdier self-dual. Moreover, since Verdier duality commutes with the perverse vanishing cycle functor ${^p}\varphi_{\of}=\varphi_{\of}[-1]$, we have
\begin{equation}\label{tr8}
\begin{split}
\cD \left(\varphi_{\of} j_!\Q_X[n] \right) &\cong 
\cD \left({^p}\varphi_{\of} j_!\Q_X[n+1] \right) \cong
{^p}\varphi_{\of} Rj_* \cD \left(\Q_X[n+1] \right) \\ &\cong
{^p}\varphi_{\of} Rj_* \Q_X[n+1] 
\cong \varphi_{\of} Rj_*\Q_{X}[n].
\end{split}
\end{equation}
In particular, the perverse sheaves $\varphi_{\of} j_!\Q_X[n]$ and $\varphi_{\of} Rj_*\Q_{X}[n]$ have the same support, which is zero-dimensional by assumption. Thus a point $p\in \Sigma^\infty \subset \oF_0 \cap H_\infty$ is an isolated point in the support of these perverse sheaves, hence the stalk and costalk cohomologies of any of these perverse sheaves at $p$ coincide, and they are concentrated in degree zero. In particular, if $i_p:\{p\} \to F_0$ is the point inclusion, then
\begin{equation}\label{tr9}
i_p^* \varphi_{\of} j_!\Q_X[n] \cong i_p^! \varphi_{\of} j_!\Q_X[n]
\end{equation}
have cohomology concentrated in degree $0$. We then have the following sequence of isomorphisms:
\begin{equation}\label{tr10}
\begin{split}
\cH^n(\varphi_{\of}\bQ_{\oX})_p & \cong H^n(i_p^*\varphi_{\of}\bQ_{\oX}) \overset{\eqref{tr7}}{\cong} H^n(i_p^*\varphi_{\of} j_!\Q_X) 
\overset{\eqref{tr9}}{\cong} H^0(i_p^!\varphi_{\of} j_!\Q_X[n]) \\
&\overset{\eqref{tr8}}{\cong} H^0(i_p^! \cD\left(\varphi_{\of} Rj_*\Q_{X}[n] \right)) 
\cong H^0( \cD\left( i_p^* \varphi_{\of} Rj_*\Q_{X}[n] \right)) \\
& \cong H^0(  i_p^* \varphi_{\of} Rj_*\Q_{X}[n] )^\vee
\cong H^n(  i_p^* \varphi_{\of} Rj_*\Q_{X})^\vee \\
& \cong \left(\cH^n(\varphi_{\of}Rj_*\bQ_X)_p\right)^\vee ,
\end{split}
\end{equation}
which proves the first equality of \eqref{mfi}. The second equality in \eqref{mfi} is well known (e.g., see \cite[Ex.10.4.15]{MS}).
\end{proof}

% that one has for $k \in \bZ$, \be\label{mfi} \cH^k(\varphi_{\of-b}Rj_*\bQ_X)_p \cong \cH^k(\varphi_{\of-b}\bQ_{\oX})_p \cong \widetilde{H}^k(MF_{\of-b,p};\bQ),\ee with $MF_{\of-b,p}$ denoting the Milnor fiber of $\of-b$ at the point $p\in \Sigma^\infty_b$. $\hfill\square$

\begin{remark}
The first equality in \eqref{mfi2} was also stated in \cite[Ex.6.2.20]{Di}; however, the argument given there appears to contain a gap (specifically, in the use of Verdier duality). $\hfill\qed$
\end{remark}

%An immediate consequence of Lemma \eqref{identif} is the following.
%\begin{corollary}\label{identif2}
%\end{corollary}

In what follows, we  investigate the non-vanishing range of the vanishing cohomology $V^k(b)$ at $b\in B$, that is, for $k \geq n-s_b$.
As the case $s_b=0$ was already considered in Corollary \ref{iso}, here we assume that $s_b>0$.

Denote by $\cS_b$ a Whitney stratification of ${\oF}_b$ so that $\varphi_{\overline{f}-b}Rj_*\bQ_X$ is $\cS_b$-constructible and $\Sigma_b$ is the union of strata of complex dimension $\leq s_b$. Upon refining $\cW$, we can assume that any strata in $\cS_b$ are also strata of $\cW$. Let us make the following notations:
\begin{itemize}
\item[] $\Sigma_{b}^{\ell}=$ the union of $\ell$-dimensional strata of $\cS_b$, which are contained in $\Sigma_b$;
\item[] $U_{b}^{\ell}=$ the union of strata of $\cS_b$ of dimension $\geq \ell$, which are contained in $\Sigma_b$.
\end{itemize}
Then each $U_{b}^{\ell}$ is an open subset of $\Sigma_b$, and 
\be\label{of}\Sigma_{b}^{s_b}=U_{b}^{s_b} \subseteq U_{b}^{s_b-1} \subseteq \cdots \subseteq U_{b}^0=\Sigma_b\ee
with $\Sigma_{b}^{\ell} = U_{b}^{\ell} \setminus U_{b}^{\ell+1}$. 
With $\cP_b$ as in \eqref{pp}, set $$\cP_b^{\ell}:=\cP_b\vert_{U_{b}^{\ell}}$$ with $\cP_b^0=\cP_b$, and note that $\cP_b^{\ell}$ is a perverse sheaf of $U_{b}^{\ell}$. 

We are now ready to prove the following result, which can be viewed as the global counterpart of \cite[Thm.3.1]{MPT2}:
\bt\label{t14a}
Let $f\colon \bC^{n+1} \to \bC$ be a nonconstant polynomial function, $b\in B$ a bifurcation value of $f$, and assume $s_b>0$. Then the following assertions hold:
\begin{itemize}
\item[(a)]  for any $j=0,\ldots, s_b-1$ there is a monomorphism
\be\label{newthad}
V^{n-s_b+j}(b) \hookrightarrow \bH^{-s_b+j}(U_{b}^{s_b-j};\cP_b^{s_b-j}).
\ee
\item[(b)]  if $s_b\geq 2$, then for any $j=0,\ldots,s_b-2$ there is an isomorphism
\be\label{newthad2}
V^{n-s_b+j}(b) \cong \bH^{-s_b+j}(U_{b}^{s_b-j-1};\cP_b^{s_b-j-1}).
\ee
\end{itemize}
\et

\begin{proof}
(a) \  Let us fix an integer $j=0,\ldots, s_b-1$.  For any $0 \leq \ell \leq s_b-j-1$ and using the above notations, consider the inclusions
$$\Sigma_{b}^\ell \overset{v_\ell}{\hookrightarrow} U_{b}^\ell \overset{u_\ell}{\hookleftarrow} U_{b}^{\ell+1}$$
and the attaching distinguished triangle on $U_{b}^\ell$
\be\label{stl}{v_\ell}_!{v_\ell}^!\cP_{b}^\ell \lra \cP_{b}^\ell \lra R{u_\ell}_*{u_\ell}^*\cP_{b}^\ell \overset{[1]}{\lra}\ee
with ${v_\ell}_!={v_\ell}_*$ and ${u_\ell}^*\cP_{b}^\ell \cong \cP_{b}^{\ell+1}$. 
The hypercohomology long exact sequence associated to \eqref{stl} contains the terms
\be\label{stl2} \cdots \to \bH^{-s_b+j}(\Sigma_{b}^\ell; {v_\ell}^!\cP_b^\ell) \to \bH^{-s_b+j}(U_b^\ell;\cP_b^\ell) \to   \bH^{-s_b+j}(U_b^{\ell+1};\cP_b^{\ell+1}) \to \cdots\ee
The group $\bH^{-s_b+j}(\Sigma_b^\ell; {v_\ell}^!\cP_b^\ell)$ is computed by a hypercohomology spectral sequence whose $E_2$-term is given by
\be\label{spsup} E_2^{p,q}=H^p(\Sigma_b^\ell; \cH^q({v_\ell}^!\cP_b^\ell)).\ee
The complex ${v_\ell}^!\cP_b^\ell$ is  constructible on $\Sigma_b^\ell$, hence its cohomology sheaves are local systems on every connected component of $\Sigma_b^\ell$. By reasons of dimension, we then get that 
$E_2^{p,q}=0$ if $p<0$ or $p>2\ell$. On the other hand, the costalk condition for the perverse sheaf $\cP_b^\ell$ on $U_b^\ell$ (with the induced stratification) yields that $\cH^q({v_\ell}^!\cP_b^\ell) \cong 0$ for all $q < -\ell$. Therefore, $E_2^{p,q}=0$ if $q<-\ell$. 
Altogether, since $-s_b+j<-\ell$, we get that $E_2^{p,q}=0$ for any pair of integers $(p,q)$ with $p+q=-s_b+j$. The spectral sequence \eqref{spsup} then implies that 
\be\label{stl3}\bH^{-s_b+j}(\Sigma_\ell; {v_\ell}^!\cP_b^\ell)\cong 0.\ee
Back to \eqref{stl2}, in view of \eqref{stl3} we get an injective map
\be \bH^{-s_b+j}(U_b^\ell;\cP_b^\ell) \hookrightarrow   \bH^{-s_b+j}(U_b^{\ell+1};\cP_b^{\ell+1})\ee for any $0 \leq \ell \leq s_b-j-1$.
Together with \eqref{vanb} and recalling that $\Sigma_b=U_b^0$,  by iterating the above procedure one gets a composition of monomorphisms
$$V^{n-s_b+j}(b) \cong \bH^{-s_b+j}(U_b^0;\cP_b) \hookrightarrow  \bH^{-s_b+j}(U_b^1;\cP_b^1) \hookrightarrow \cdots \hookrightarrow  \bH^{-s_b+j}(U_b^{s_b-j};\cP_b^{s_b-j}).$$
This completes the proof of \eqref{newthad}.

 (b) \  Let us now assume that $s_b \geq 2$ and fix an integer $j=1,\ldots,s_b-1$. 

In view of \eqref{stl3}, the long exact sequence \eqref{stl2} contains the terms:
\be\label{stl2b} \cdots \to \bH^{-s_b+j-1}(\Sigma_b^\ell; {v_\ell}^!\cP_b^\ell) \to \bH^{-s_b+j-1}(U_b^\ell;\cP_b^\ell) \to   \bH^{-s_b+j-1}(U_b^{\ell+1};\cP_b^{\ell+1}) \to 0 \to \cdots\ee
For any $0 \leq \ell \leq s_b-j-1$, the same arguments used for studying the spectral sequence \eqref{spsup} yield that $\bH^{-s_b+j-1}(\Sigma_\ell; {v_\ell}^!\cP_b^\ell) \cong 0$ since $-s_b+j-1<-\ell$. In particular, \eqref{stl2b} yields isomorphisms
$$\bH^{-s_b+j-1}(U_b^\ell;\cP_b^\ell) \cong   \bH^{-s_b+j-1}(U_b^{\ell+1};\cP_b^{\ell+1}) $$
for all $0 \leq \ell \leq s_b-j-1$. Together with \eqref{vanb}, this then yields isomorphisms
\be\label{rei}V^{n-s_b+j-1}(b) \cong \bH^{-s_b+j-1}(U_b^0;\cP_b) \cong  \bH^{-s_b+j-1}(U_b^1;\cP_b^1) \cong \cdots \cong  \bH^{-s_b+j-1}(U_b^{s_b-j};\cP_b^{s_b-j}).\ee
The isomorphism \eqref{newthad2} is then obtained by reindexing (i.e., replacing $j$ by $j+1$ in \eqref{rei}).
\end{proof}

An immediate consequence of Theorem \ref{t14a}(b) is the following.  
\begin{corollary}\label{c2.2}
If $s_b \geq 2$, then for any $j=0,\ldots, s_b-2$, the group $V^{n-s_b+j}(b)$ depends only on the strata of dimension $\geq s_b-j-1$ in $\Sigma_b$.
\end{corollary}

\begin{remark}\label{dif} 
Assuming $s_b\geq 2$ and fixing  $j=0,\ldots, s_b-2$, if there are no strata of dimension $s_b-j-1$ in $\Sigma_b$, then $U_b^{s_b-j}=U_b^{s_b-j-1}$, so in this case \eqref{newthad2} is a finer result than \eqref{newthad}. In general, the right-hand side of either \eqref{newthad} or \eqref{newthad2} can be computed via the hypercohomology spectral sequence, though explicit computations can be tedious; see below for a special case. $\hfill\square$ \end{remark}

In what follows, we specialize Theorem \ref{t14a} to the case $j=0$, to derive more explicit information about $V^{n-s_b}(b)$, that is, the lowest (possibly non-trivial) vanishing cohomology group at the bifurcation value $b \in B$ of $f$.
We first introduce some notations.

Recall from \eqref{of} that $$U_b^{s_b}=\Sigma_b^{s_b}=\bigsqcup_i \Sigma_b^{s_b,i},$$
where $\Sigma_b^{s_b,i}$ are the $s_b$-dimensional (connected) strata of $\Sigma_b$. With $\cP_b$ as in \eqref{pp}, set
\begin{equation}\label{pp11}
\cP_b^{s_b,i}:=\cP_b\vert_{\Sigma_b^{s_b,i}},
\end{equation}
and we have the following.
\begin{lemma}\label{identif3}
In the above notations, we have
\begin{equation}\label{pp12}
\cP_b^{s_b,i} \cong \mathcal{L}_b^{s_b,i}[s_b],
\end{equation}
where $\mathcal{L}_b^{s_b,i}$ is a local system on $\Sigma_b^{s_b,i}$ whose stalk ${L}_b^{s_b,i}$ at a point $p_{s_b,i} \in \Sigma_b^{s_b,i}$ is a $\Q$-vector space of dimension \begin{center} $\mu^\pitchfork_{s_b,i}:= \dim \widetilde{H}^{n-s_b}(MF^\pitchfork_{s_b,i};\bQ)$,\end{center}
with $MF^\pitchfork_{s_b,i}$ the transversal Milnor fiber of $\of-b$ along $\Sigma_b^{s_b,i}$.
\end{lemma}

\begin{proof}
Formula \eqref{pp12} follows by constructibility, since $\cP_b^{s_b,i}$ is a perverse sheaf on the $s_b$-dimensional stratum $\Sigma_b^{s_b,i}$. 

In order to identify the stalk of the local system $\mathcal{L}_b^{s_b,i}$, let $\oE \subset \C P^{n+1}$ be a general subspace of codimension $s_b$, intersecting $\Sigma_b^{s_b,i}$ transversally at the point $p_{s_b,i}$. Denote by $E=\oE \cap \{x_0=1\}\subset \C^{n+1}$ the corresponding affine linear subspace and let $g:=f|_E$. Then the closure $\overline{Y}$ of the graph $Y$ of $g$ can be identified with $\oX \cap (\overline{E} \times \C)$, and the proper extension $\overline{g}$ of $g$ can be identified with the restriction of $\of$ to $\oY$. Note that if $G_b=g^{-1}(b)$, with projective closure $\overline{G}_b=\overline{g}^{-1}(b)$, then $\overline{G}_b=\overline{F}_b \cap \overline{E}$. Denote the relevant inclusion maps by the arrows in the following diagram (with $X\cong \C^{n+1}$ the graph of $f$):
\[
\xymatrix{
\oG_b=\oF_b \cap \oE \ar[d]_{\oi_b}\ar[r] &\overline{Y}=\oX \cap (\overline{E} \times \C) \ar[d]^{\oi} & Y=X \cap E \ar[l]_{j_E}  \ar[d]^{i}\\
\oF_b \ar[r]& \oX & X \ar[l]^{j}
}
\]
The base change  formula for vanishing cycles (e.g., see \cite[Lem.4.3.4]{Sch}) yields that
\be\label{1000} 
\begin{split} 
\oi_b^*\varphi_{\of-b}(Rj_*\bQ_X) \cong 
\varphi_{\og-b} (\oi^*Rj_*\Q_X) \cong \varphi_{\og-b}  (Rj_{E*}{\bQ}_Y),
\end{split}\ee
where  the second isomorphism uses the base change formula from \cite[Prop.4.3.1]{Sch}.

By transversality, the function $\og-b$ has an {\it isolated} (stratified) singularity at $p_{s_b,i}$, and hence the point $p_{s_b,i}$ is an isolated point in the support of the 
perverse sheaf $\varphi_{\og-b} (Rj_{E*}{\bQ}_{Y}[n-s_b])$. Therefore, the stalk 
%and costalk 
cohomology of $\varphi_{\og-b} (Rj_{E*}{\bQ}_{Y}[n-s_b])$ at $p_{s_b,i}$ 
%are isomorphic, and they are 
is concentrated in degree $0$, or, equivalently, the stalk cohomology of $\cP_b^{s_b,i}$ is concentrated in degree $-s_b$. 
Moreover, 
%, where they are free. 
formulae \eqref{mfa} and  \eqref{mfi}  identify this stalk as a $\Q$-vector space of dimension $\dim \widetilde{H}^{n-s_b}(MF^\pitchfork_{s_b,i};\bQ)$, with $MF^\pitchfork_{s_b,i}$ the Milnor fiber of $\og-b$ at $p_{s_b,i}$ (i.e., the transversal Milnor fiber of $\of-b$ at this point). This concludes the proof.
\end{proof}

\begin{remark}\label{tmn}
As in \cite{MPT2}, we call $\mu^\pitchfork_{s_b,i}$ the {\it transversal Milnor number} of $\of-b$ along $\Sigma_b^{s_b,i}$, and
note that by transversality one has that $\mu^\pitchfork_{s_b,i}=\dim \widetilde{H}^{n-s_b}(MF_{s_b,i};\bQ)$,
with $MF_{s_b,i}$ the Milnor fiber of $\of-b$ at a point in $\Sigma_b^{s_b,i}$. \hfill\qed
\end{remark}

For any $i$, denote by
\be\label{vm}
 A_b^i:\pi_1(\Sigma_b^{s_b,i}) \lra {\rm Aut}\big({L}_b^{s_b,i}\big),
 \ee
the action of the fundamental group $\pi_1(\Sigma_b^{s_b,i})$ of the stratum $\Sigma_b^{s_b,i}$ on the stalk ${L}_b^{s_b,i}$ of the local system $\mathcal{L}_b^{s_b,i}$ of Lemma \ref{identif3}, and refer to $A_b^i$ as the {\it local system monodromy} along the stratum $\Sigma_b^{s_b,i}$.

With the above notations and under the hypotheses of Theorem \ref{t14a}, we deduce the following.

\bt\label{c14a}
For $b \in B$ a bifurcation value of $f$ with $s_b>0$, there is a monomorphism  \footnote{Here, the $A_i$-invariant part of ${L}_b^{s_b,i}$ is the intersection of the invariant subspaces over some set of generators of $\pi_1(\Sigma_b^{s_b,i})$.}
\be\label{upb}
V^{n-s_b}(b) \hookrightarrow \bigoplus_i \left({L}_b^{s_b,i}\right)^{A_b^i}.
\ee
In particular, 
 \be\label{2.18}
\dim V^{n-s_b}(b) \leq \sum_i \dim \left({L}_b^{s_b,i}\right)^{A_b^i} \leq \sum_i  \mu^\pitchfork_{s_b,i}.
 \ee
 
  If, moreover, $s_b \geq 2$ and $\Sigma_b$ does not contain any strata of dimension $s_b-1$, then \eqref{upb} is an isomorphism and  the first inequality in \eqref{2.18} becomes an equality. 
\et

\begin{proof}
First note that \eqref{newthad} yields a monomorphism
\be\label{monob}
V^{n-s_b}(b) \hookrightarrow \bH^{-s_b}(\Sigma_b^{s_b};\cP_b^{s_b}).
\ee
The hypercohomology spectral sequence together with the support condition for perverse sheaves then yield in the notations of Lemma \ref{identif3} that (e.g., see \cite[Proposition 5.2.20]{Di})
\be\label{monobb} 
\begin{split}
\bH^{-s_b}(\Sigma_b^{s_b};\cP_b^{s_b}) &\cong H^0(\Sigma_b^{s_b};\cH^{-s_b}(\cP_b^{s_b}))
\cong \bigoplus_i H^{0}(\Sigma_b^{s_b,i};\cH^{-s_b}(\cP_b^{s_b,i})) \\
& \cong \bigoplus_i H^{0}(\Sigma_b^{s_b,i};\mathcal{L}_b^{s_b,i}) \cong \bigoplus_i \left({L}_b^{s_b,i}\right)^{A_b^i}.
\end{split}
\ee
Altogether, this proves \eqref{upb}, while \eqref{2.18} follows by computing dimensions in \eqref{upb}.

Let us next assume that $s_b\geq 2$. By setting $j=0$ in \eqref{newthad2} we get an isomorphism
\be\label{usm1}
V^{n-s_b}(b) \cong \bH^{-s_b}(U_{b}^{s_b-1};\cP_b^{s_b-1}).
\ee
Consider the inclusions
$$\Sigma_b^{s_b-1} \overset{\alpha}{\hookrightarrow} U_b^{s_b-1} \overset{\beta}{\hookleftarrow} U_b^{s_b}=\Sigma_b^{s_b}$$
(or, in the notations of Theorem \ref{t14a}, $\alpha=v_{s_b-1}$ and $\beta=u_{s_b-1}$)
and the corresponding attaching triangle 
\be\label{stls1}{\alpha}_!{\alpha}^!\cP_b^{s_b-1} \to \cP_b^{s_b-1} \to R{\beta}_*{\beta}^*\cP_b^{s_b-1} \overset{[1]}{\to}\ee
with ${\alpha}_!={\alpha}_*$ and ${\beta}^*\cP_b^{s_b-1} \cong \cP_b^{s_b}$. In view of \eqref{usm1}, the hypercohomology long exact sequence associated to \eqref{stls1} contains the terms
\be\label{stl2s} \cdots \to \bH^{-s_b}(\Sigma_b^{s_b-1}; {\alpha}^!\cP_b^{s_b-1}) \to V^{n-s_b}(b)  \to   \bH^{-s_b}(U_b^{s_b};\cP_b^{s_b}) \to \bH^{-s_b+1}(\Sigma_b^{s_b-1}; {\alpha}^!\cP_b^{s_-1}) \to \cdots\ee
If $\Sigma_b$ does not contain any strata of dimension $s_b-1$, then $\Sigma_b^{s_b-1}=\emptyset$, so \eqref{stl2s} shows that in this case \eqref{monob} becomes an isomorphism. The remaining assertions follow now immediately from \eqref{monobb}.
\end{proof}

%%%%%%%%%%%%%%%%

\section{On the cohomology of the general fiber}\label{genf}

Recall  from the previous section that for $b \in B$ a bifurcation value of $f$ we denote
\[\Sigma_b:={\rm Supp} \ \varphi_{\overline{f}-b}Rj_*\bQ_X,\] and $s_b:=\dim \Sigma_b$.
Moreover, we let 
\begin{center} $\Sigma^a_b:=\Sigma_b \cap X=\Sing(f) \cap F_b$ and $\Sigma_b^\infty:=\Sigma_b \cap \oX_\infty$,\end{center} so that $\Sigma_b=\Sigma^a_b \cup \Sigma_b^\infty$.

Let us introduce the notation \be\label{sdef} s:=s(\of):=\max_{b\in B} s_b.\ee We assume $s \geq 0$ and we say that ``$f$ has $s$-dimensional singularities, including at infinity''.  
Note that, if we choose a Whitney stratification $\cW$ of $\oX$ with respect to which $Rj_*\bQ_X$ is constructible and let $s(\of,\cW):=\dim \ \Sing_{\cW}(\of)$, then \eqref{sup} yields that $s \leq s(\of,\cW)$.

Let us also set
\[B_s:=\{b \in B \mid s_b=s\}.\]
In view of the isomorphism \eqref{iso-van}, Corollary \ref{cor32} yields the following well-known vanishing result (\cite{NN,ST,Bre,Di}), together with a description of the first (possibly) non-vanishing cohomology group of the general fiber of $f$:
\bc\label{c26}
If $F$ denotes the general fiber of the polynomial map $f\colon \bC^{n+1}\to \bC$, then
$\widetilde{H}^k(F;\bQ)\cong 0$ for all integers $k \notin [n-s,n]$. Moreover,
\be\label{nmsgen}
\widetilde{H}^{n-s}(F;\bQ)\cong \bigoplus_{b \in B_s} V^{n-s}(b).
\ee
\ec

Let us also mention here the special case of ``tame'' polynomials, considered in \cite{Sab,Pa2,Di}. Recall that $f\colon \bC^{n+1}\to \bC$ is called (cohomologically) {\it tame} with respect to the compactification $\of$ if the complex $Rj_*\bQ_X$ has no vanishing cycles at infinity, that is, for any $b \in B$ the complex $\varphi_{\overline{f}-b}Rj_*\bQ_X$ is supported on $\Sigma^a_b$. The following result is well-known:
\bc\label{coro32}
If $f\colon \bC^{n+1}\to \bC$ is a tame polynomial with respect to $\of$, then $\widetilde{H}^k(F;\bQ)=0$ for all $k \neq n$.
\ec

\begin{proof}
In view of  \eqref{iso-van}, it suffices to show that $V^k(b)=0$ for all $k \neq n$ and $b \in B$. 

Let $j_b:F_b\hookrightarrow \oF_b$ be the inclusion map. The tameness assumption implies via the obvious attaching triangle that
\[
\varphi_{\overline{f}-b}Rj_*\bQ_X \cong {j_b}_! j_b^* \varphi_{\overline{f}-b}Rj_*\bQ_X \cong {j_b}_!  \varphi_{{f}-b}j^*Rj_*\bQ_X \cong {j_b}_!  \varphi_{{f}-b}\bQ_X.
\]
Hence, using \eqref{vanb},
\[ \begin{split}
V^k(b)\cong \bH^k(\overline{F}_b; \varphi_{\overline{f}-b}Rj_*\bQ_X) & \cong
\bH^k(\overline{F}_b; {j_b}_!  \varphi_{{f}-b}\bQ_X) \cong \bH_c^k(F_b;   \varphi_{{f}-b}\bQ_X)\\ & \cong \bH_c^{k-n}(F_b;   \varphi_{{f}-b}\bQ_X[n]).
\end{split}
\]
Since $ \varphi_{{f}-b}\bQ_X[n]$ is a perverse sheaf on the affine variety $F_b$, Artin's vanishing theorem (e.g., see \cite[Thm.10.3.59]{MS}) yields that the last group of the above sequence of isomorphisms vanishes in negative degrees, i.e., for $k<n$.
\end{proof}

\br\label{rem33}
In fact, the tameness assumption for $f$ implies immediately that $s=0$, i.e.,  $f$ has only isolated singularities on $\bC^{n+1}$, see, e.g., \cite[Thm.6.3.17]{Di}. $\hfill\qed$
\er

In this paper, we are mainly interested in the first (possibly) non-vanishing cohomology group $H^{n-s}(F;\bQ)$ and the corresponding Betti number. The case $s=0$ was already considered in several references, including \cite{ST,Hamm,Di}. The relevant statement follows immediately from \eqref{iso-van} and \eqref{e23}, together with the identifications from \eqref{mfa} and \eqref{mfi}. In our notations, it can be phrased as follows.

\bc\label{geniso}
Let $F$ denote the general fiber of the polynomial map $f:\bC^{n+1}\to \bC$ with only isolated singularities, including at infinity (that is, $s=0$). With the above notations, 
\be\label{genisob}
b_n(F)=\sum_{b \in B} \sum_{p\in \Sigma_b} \mu_{\of-b,p},
\ee
where $\mu_{\of-b,p}$ is the Milnor number of $\of-b$ at the point $p \in \Sigma_b$. $\hfill\qed$
\ec

In this section we assume that there exists a bifurcation value $b \in B$ with $s_b>0$, hence also $s>0$.

By combining Corollary \ref{c26} with Theorem \ref{c14a}, and in the notations of the latter, we get the following upper bound on the first (possibly) non-vanishing Betti number of the general fiber of $f$:
\bc\label{nmsb}
Let $F$ denote the general fiber of the polynomial map $f\colon\bC^{n+1}\to \bC$ with $s>0$. Then 
\be\label{nmsgen2}
\widetilde{H}^{n-s}(F;\bQ)\hookrightarrow \bigoplus_{b \in B_s}  \bigoplus_i \left({L}_b^{s,i}\right)^{A_b^i}.
\ee
In particular, 
 \be\label{nmsgb}
b_{n-s}(F) \leq  \sum_{b \in B_s} \sum_i  \mu^\pitchfork_{b,s,i},
 \ee
with $\mu^\pitchfork_{b,s,i}$ denoting as before the transversal Milnor  number of $\of-b$ at an $s$-dimensional stratum in the atypical fiber $\oF_b$, for some $b \in B_s$, and $i$ indexing the connected $s$-dimensional strata in $\Sigma_b$.

 If, moreover, $s\geq 2$ and for every $b \in B_s$ the set $\Sigma_b$ has no strata of dimension $s-1$, then \eqref{nmsgen2} becomes an isomorphism. $\hfill\qed$
\ec

\begin{remark}[Transversal Milnor numbers and Milnor-number jumps at infinity]\label{calcmu}
Let $b\in B_s$, and let $\Sigma_b^{s,i}\subset\Sigma_b^\infty$ be an $s$-dimensional stratum appearing in Corollary~\ref{nmsb}. Choose a general point $p\in\Sigma_b^{s,i}$ and consider the transverse graph $\overline{g}:\overline{Y}\to\mathbb C$ constructed in the proof of Lemma~\ref{identif3}. Let $j_E:Y\hookrightarrow\overline{Y}$ be the inclusion, $\overline{G}_t=\bar g^{-1}(t)$ be the fiber of $\overline g$ over $t \in \C$, and set $m=n-s$. Let $MF_{b,s,i}^{\pitchfork}$ denote the corresponding transversal Milnor fiber, i.e., the Milnor fiber of the projection $\overline{g}-b$ on the transverse graph at $p$. The contribution of this stratum to the right-hand side of the bound in Corollary~\ref{nmsb} is the Milnor-L\^e number 
\[
 \mu_{b,s,i}^{\pitchfork}
 =\dim\widetilde H^{m}(MF_{b,s,i}^{\pitchfork};\mathbb Q).
\]
%where $MF_{b,s,i}^{\pitchfork}$ is the Milnor fiber of the projection $\bar g-b$ on the transverse graph. 
The Milnor fiber $MF_{b,s,i}^{\pitchfork}$ is represented by $\overline{G}_c\cap B_\varepsilon(p)$ for $0<|c-b|\ll\varepsilon$.
Since the transverse graph may be singular, this nearby fiber need
not be a (full) smoothing of the hypersurface germ $(\overline{G}_b,p)$.
Assume, moreover, that the germs $(\overline{G}_b,p)$ and $(\overline{G}_c,p)$ have at most isolated hypersurface singularities in the transverse ambient space, and that $p$ is the
only singular point of $\overline{G}_c\cap B_\varepsilon(p)$. Denote their usual hypersurface Milnor numbers by
\[
 \mu_p(b):=\mu(\overline{G}_b,p),\qquad
 \mu_p^{gen}:=\mu(\overline{G}_c,p),\qquad 0<|c-b|\ll1.
\]
Then
\begin{equation}\label{MilLe}
\mu_{b,s,i}^{\pitchfork}
       =\mu_p(b)-\mu_p^{gen}
\end{equation}
(see \cite[Example~6.2.20]{Di} or \cite[formula (2.2)]{ST0}). The term
$\mu_p^{gen}$ is a measure of the singularity in the nearby compactified transverse fiber.

Even though formula \eqref{MilLe} is well-known, for the benefit of the reader we include below an explanation for its derivation. The transverse choice and formula~\eqref{1000} make $p$ an isolated point of the support of the transverse vanishing cycle complex. Applying the local argument of Lemma~\ref{identif} to $\overline{g}$ gives
\[
 \left(
 \cH^m(\varphi_{\overline g-b}Rj_{E*}\mathbb Q_Y)_p
 \right)^\vee
 \cong
 \cH^m(\varphi_{\bar g-b}\mathbb Q_{\bar Y})_p
 \cong
 \widetilde H^m(MF_{b,s,i}^{\pitchfork};\mathbb Q).
\]
This is (a mild correction of) the comparison of stalk dimensions used in \cite[Example~6.2.20]{Di}. To obtain the Milnor-number jump from it, let $\overline G_t=\{h_t=0\}$ in local coordinates near $p$, and set $M_c:=\{h_c=0\}\cap B_\varepsilon(p)$. For sufficiently small generic $\delta\ne0$, the smooth space
\[
 N_{c,\delta}:=\{h_c=\delta\}\cap B_\varepsilon(p)
\]
obtained by smoothing the singularity of $M_c$ at $p$ has the
Milnor-fiber topology of the hypersurface germ $(\overline G_b,p)$.
Hence \[\chi(N_{c,\delta})=1+(-1)^m\mu_p(b).\] 
On the other hand, smoothing the singularity at $p$ of $M_c$  changes the 
Euler characteristic by $(-1)^m\mu_p^{gen}$. Indeed, the singular germ in a smaller Milnor ball has Euler characteristic $1$, whereas its smoothing has Euler characteristic
$1+(-1)^m\mu_p^{gen}$. Outside that smaller ball the spaces are identified by local triviality.
Therefore
\[
 \chi(N_{c,\delta})
 =\chi(M_c)+(-1)^m\mu_p^{gen},
\]
which implies
\[
 \widetilde\chi(M_c)
 =(-1)^m\bigl(\mu_p(b)-\mu_p^{gen}\bigr).
\]
By the isolated-support and duality argument in Lemma~\ref{identif}, applied to $\overline g$, the reduced cohomology of $M_c$ is concentrated in degree $m$. Hence
\[
 \mu_{b,s,i}^{\pitchfork}
 =\dim\widetilde H^m(M_c;\mathbb Q)
 =\mu_p(b)-\mu_p^{gen},
\]
as claimed.

More generally, if the nearby hypersurface has several isolated singular points $p_i$ in the chosen Milnor ball, the same calculation replaces $\mu_p^{gen}$ by $\sum_i\mu(h_c,p_i)$. The additional isolated hypersurface singularity hypotheses above are needed to express the answer by ordinary Milnor numbers, and they are not additional assumptions on Corollary~\ref{nmsb} itself. $\hfill\qed$
\end{remark}

\br\label{whp2}
If $f$ is a homogeneous polynomial of degree $d \geq 2$, then $c=0$ is its only atypical value, and the general fiber $F$ coincides with the Milnor fiber of $f$ (while $F_0$ is contractible). Moreover, since $\oF_0$ intersects the hyperplane at infinity $H_\infty$ transversally (in a stratified sense), all top dimensional strata appearing in the statement of Corollary \ref{nmsb} arise from the affine part. Hence, in this setting, Corollary \ref{nmsb} agrees with the corresponding bounds for the first nonvanishing Betti number of Milnor fibers established in \cite[Thm.3.4(a)]{MPT2}.  $\hfill\qed$
\er

\br\label{whp3}
Assume that $f\colon\bC^{n+1}\to \bC$ is weighted homogeneous of weighted degree $D$ with respect to positive integral weights $\mathbf w=(w_1,\ldots,w_{n+1})$, and that $s:=\dim_{\mathbb C}\Sing(f)>0$. Then, as already noted in Example~\ref{whp}, $0$ is the only atypical value, the general fiber $F$ is the Milnor fiber of the germ $(f,0)$, and $F_0=f^{-1}(0)$ is contractible.
Consider the weighted graph compactification
\[
\overline X_{\mathbf w}
=
\left\{
([x_0:x_1:\cdots:x_{n+1}],t)
\in\mathbb P(1,w_1,\ldots,w_{n+1})\times\mathbb C
\;\middle|\;
f(x_1,\ldots,x_{n+1})=t x_0^D
\right\}.
\]
Let \(\overline f_{\mathbf w}\) be projection onto the second factor,
let $j:X=\{x_0\neq0\}\hookrightarrow\overline X_{\mathbf w}$, and put
\[
\Sigma_{0,\mathbf w}
:=
\operatorname{Supp}
\varphi_{\overline f_{\mathbf w}}
Rj_*\mathbb Q_X
\quad \text{and} \quad 
\Sigma^\infty_{0,\mathbf w}
:=
\Sigma_{0,\mathbf w}\cap\{x_0=0\}.
\]
Note that (by the weighted Euler identity) $\Sing(f) \subset F_0$. Moreover, one also has that \[
\Sigma^\infty_{0,\mathbf w}
\subset
\mathbb P_{\mathbf w}(S)
:=
(S\setminus\{0\})/\mathbb C^*, 
\]
and therefore $\dim_{\mathbb C}\Sigma^\infty_{0,\mathbf w} \leq s-1<s$ when $s>0$, with $\Sigma^\infty_{0,\mathbf w}=\emptyset$ if $s=0$.

As already indicated at the end of Introduction, a formula similar to \eqref{nmsgen2} can be obtained by replicating the proof of Corollary~\ref{nmsb} in the context of a partial compactification $\oX$ of $\bC^{n+1}\cong X$ endowed with a proper extension of $f$, as long as the inclusion $j:X \hookrightarrow \oX$ is affine. A priori, the only thing that may differ in its formulation is the stalks ${L}_b^{s,i}$ of the local systems from Lemma \ref{identif3} along the $s$-dimensional strata in the infinity part of $\Sigma_b$, for some $b \in B_s$.
However, as already pointed out above, in the setting of our weighted graph compactification, all strata in $\Sigma^\infty_{0,\mathbf w}$ have dimension $<s$, while 
\[
\Sigma^a_{0,\mathbf w}=\operatorname{Sing}(f)
\]
has dimension $s$. Therefore every $s$-dimensional stratum
contributing to formula \eqref{nmsgen2} adapted to our compactification lies in the affine part. Therefore, Corollary~~\ref{nmsb} gives
\begin{equation}\label{lastf}
\widetilde H^{\,n-s}(F;\mathbb Q)
\hookrightarrow
\bigoplus_i
\widetilde H^{\,n-s}
   (MF^\pitchfork_{s,i};\mathbb Q)^{A_i},
\end{equation}
and, in particular,
\[
b_{n-s}(F)
\leq
\sum_i\mu^\pitchfork_{s,i},
\]
which is the corresponding bound of \cite[Thm.3.4(a)]{MPT2}. $\hfill\qed$
\er

%%%%%%%%%%%%%%%%%%%%

\section{On the cohomology of atypical fibers}\label{atyp}
In this section, we describe vanishing results for the cohomology of an atypical fiber $F_b$ of $f$, for $b \in B$, as well as an upper bound on the first (possibly) non-vanishing Betti number of $F_b$, with no restriction on the dimension of the singularities. Some of these results are already contained in \cite[Sections 6.2, 6.3]{Di} in the case of isolated singularities, including at infinity, and the vanishing results (Theorem \ref{tat} below) appear in an even more general setting in \cite{Bre}. %We generalize the result on the bound to non-isolated singularities. 
We continue using  the notations from Section \ref{genf}.

We begin with the following result, proved in \cite[Thm.3.1]{Bre} (see also \cite[Thm.6.3.23(iii)]{Di} for the case $s=0$). We include its proof here for  completeness, as it also sets the stage for further considerations in this section.  

\bt\label{tat}
Assume that the complex polynomial $f\colon \bC^{n+1} \to \bC$ has $s$-dimensional singularities, including at infinity. Let $b \in B$ be a bifurcation value of $f$, with corresponding atypical fiber $F_b=f^{-1}(b)$. Then 
\begin{center}
$\wti{H}^k(F_b;\bQ)=0$ \ \ for all \ $k < n-s-1$.
\end{center}
\et

\begin{proof} As before, we denote by $X \cong \C^{n+1}$ the graph of $f$, and $\oX$ its closure in $\C P^{n+1} \times \C$. 
Recall that $s_b:=\dim \Sigma_b$, where $\Sigma_b={\rm Supp} \ \varphi_{\overline{f}-b}Rj_*\bQ_X$, and $s=\max_{b\in B} s_b$. We set
\[F_b^\infty:=\oF_b \cap \oX_\infty,\]
and consider the following commutative diagram of inclusions:
\[
\xymatrix{
F_b \ar[d]_{k_b}\ar[r]^{j_b}& \oF_b \ar[d]^{\ok_b} & \ar[d]^{k^\infty_b}\ar[l]_{i_b} F_b^\infty \\
X  \ar[r]_{j}& \oX & \ar[l]^{i} \oX_\infty,
}
\]
%When working with vanishing cycles with support on $\Sigma_b$, we will  use the same symbols for the inclusions $j_b: \Sigma^a_b:=\Sigma_b \cap X \hookrightarrow \Sigma_b$ and $i_b: \Sigma_b^\infty:=\Sigma_b \cap \oX_\infty \hookrightarrow \Sigma_b$.

Let $T(F_b)$ denote as before a small tube around $F_b$. Note that  \eqref{ses} and Corollary \ref{c26} yield that 
\be\label{tvan}
\wti{H}^k(T(F_b);\bQ)=0, \ \ {\rm for \ all} \ k<n-s.
\ee
From the (reduced) cohomology long exact sequence for the pair $(T(F_b),F_b)$, we see that in order to prove the assertion in the theorem it suffices to show that \be\label{want} H^k(T(F_b),F_b;\bQ)=0, \ \  {\rm for \ all} \ k<n-s.\ee

The cohomology of $T(F_b)$ is computed by
\[
{H}^k(T(F_b);\bQ)\cong \cH^k(Rf_*\bQ_X)_b= \cH^k(R\of_*Rj_*\bQ_X)_b\cong \bH^k(\oF_b;\ok_b^*Rj_*\bQ_X),
\]
where the last isomorphism uses the properness of $\of$. Moreover,
\[
H^k(F_b;\bQ)\cong \bH^k(F_b;k_b^*\bQ_X)\cong \bH^k(\oF_b;R{j_b}_*k_b^*\bQ_X) \cong \bH^k(\oF_b;R{j_b}_*j_b^*\ok_b^*Rj_*\bQ_X),
\]
where the last isomorphism uses the identification $j_b^*\ok_b^*Rj_*=k_b^*j^*Rj_*=k_b^*$. So the cohomology restriction homomorphism ${H}^k(T(F_b);\bQ) \to H^k(F_b;\bQ)$ is induced by the adjunction morphism
\[ \cF \lra R{j_b}_*j_b^*\cF,\]
where $\cF:=\ok_b^*Rj_*\bQ_X$, and the attaching triangle
\[ {i_b}_! i_b^! \lra id \lra R{j_b}_*j_b^* \lra \]
yields that
\be\label{e42}
H^k(T(F_b),F_b;\bQ) \cong \bH^k(\oF_b; {i_b}_! i_b^! \cF) \cong \bH^k(F_b^\infty; i_b^! \cF).
\ee

Next, apply the functor $i_b^!$ to  the distinguished triangle
\be\label{var}
\ok_b^!Rj_*\bQ_X[n+1] \lra {^p\varphi}_{\of-b}Rj_*\bQ_X[n+1] \overset{var}{\lra} 
{^p\psi}_{\of-b}Rj_*\bQ_X[n+1] \overset{[1]}{\lra}
\ee
where ${^p\varphi}:=\varphi[-1]$ and ${^p\psi}:=\psi[-1]$ are the perverse vanishing and nearby functors, and $var$ is the variation morphism. Note that 
$ i_b^! \ok_b^!={k_b^\infty}^! i^!,$
hence \[ i_b^! \ok_b^! Rj_*={k_b^\infty}^! i^! Rj_* =0, \] since $i^! Rj_* =0$.
It follows that one has an isomorphism
\be\label{variso} var: i_b^! {^p\varphi}_{\of-b}Rj_*\bQ_X[n+1] \overset{\cong}{\lra} i_b^! {^p\psi}_{\of-b}Rj_*\bQ_X[n+1]. \ee

Applying now the functor $i_b^!$ to the distinguished triangle
\be\label{can}
{^p\psi}_{\of-b}Rj_*\bQ_X[n+1]  \overset{can }{\lra} {^p\varphi}_{\of-b}Rj_*\bQ_X[n+1]
 \lra \ok_b^*Rj_*\bQ_X[n+1] \overset{[1]}{\lra}
\ee
and using \eqref{variso} together with the fact that $can \circ var=M_b -id$, where $M_b$ is the monodromy automorphism of ${^p\varphi}_{\of-b}Rj_*\bQ_X[n+1]$, we get a triangle
\be\label{ibcan}
i_b^!{^p\varphi}_{\of-b}Rj_*\bQ_X[n+1]  \overset{M_b-id}{\lra} i_b^!{^p\varphi}_{\of-b}Rj_*\bQ_X[n+1]
 \lra i_b^!\cF[n+1] \overset{[1]}{\lra}
\ee
To simplify the writing, in what follows we let
\[ \cG:={^p\varphi}_{\of-b}Rj_*\bQ_X[n+1] ,\]
which is a perverse sheaf on $\oF_b$ with $s_b$-dimensional support $\Sigma_b$.
The cohomology long exact sequence associated to \eqref{ibcan} reads as
\be\label{s47}
\cdots \to \bH^k(F_b^\infty; i_b^! \cG) \overset{M_b-id}{\to} \bH^k(F_b^\infty; i_b^! \cG) \to \bH^{k+n+1}(F_b^\infty; i_b^! \cF) \to \bH^{k+1}(F_b^\infty; i_b^! \cG) \to \cdots
\ee

To show the vanishing \eqref{want} of $H^*(T(F_b),F_b;\bQ)\cong \bH^{*}(F_b^\infty; i_b^! \cF)$ in the desired range, we consider the distinguished triangle
\be\label{ge}
{i_b}_!i_b^! \cG \lra \cG \lra R{j_b}_* j_b^* \cG \lra
\ee
and its associated hypercohomology long exact sequence. 
%Since  $j_b$ is a quasi-finite affine morphism, the functor $R{j_b}_*$ is t-exact. 
Since $j_b^*$ is t-exact and $\cG$ is perverse, it follows that $j_b^* \cG$ is perverse on $F_b$. Since the support of this perverse sheaf has dimension $\leq s_b\leq s$, the hypercohomology spectral sequence yields that 
\begin{center} $\bH^k(\oF_b; R{j_b}_* j_b^* \cG)\cong \bH^k(F_b;  j_b^* \cG)=0$ for all $k<-s$.\end{center} 
Moreover, by Corollary \ref{cor32}, $\bH^k(\oF_b; \cG)=V^{k+n}(b)=0$  for $k<-s_b$, hence it vanishes also for $k<-s$. Altogether, the hypercohomology long exact sequence for \eqref{ge} implies that 
\be\label{e44} \bH^k(F_b^\infty; i_b^! \cG)=0, \ \  {\rm for \ all} \ k<-s.\ee
Back in \eqref{s47}, this further gives
\begin{center} $\bH^k(F_b^\infty; i_b^! \cF)=0$, \ \  for all \ $k<n-s$.\end{center}
In view of \eqref{e42}, this proves the desired vanishing of \eqref{want}, thus completing the proof of the theorem.
\end{proof}

\br With $\Sigma_b$ the support of $\cG:={^p\varphi}_{\of-b}Rj_*\bQ_X[n+1]$, and $\Sigma^a_b:=\Sigma_b \cap X$ and $\Sigma_b^\infty:=\Sigma_b \cap \oX_\infty $ as before, 
consider the following commutative diagram of inclusions:
\[
\xymatrix{
F_b \ar[r]^{j_b}& \oF_b  & \ar[l]_{i_b} F_b^\infty \\
\Sigma_b^a \ar[u]^{u^a_b} \ar[r]_{\sigma_b^a}& \Sigma_b \ar[u]^{u_b} & \ar[u]_{u^\infty_b} \ar[l]^{\sigma_b^\infty} \Sigma_b^\infty.
}
\]
%When working with vanishing cycles with support on $\Sigma_b$, we will  use the same symbols for the inclusions $j_b: \Sigma^a_b:=\Sigma_b \cap X \hookrightarrow \Sigma_b$ and $i_b: \Sigma_b^\infty:=\Sigma_b \cap \oX_\infty \hookrightarrow \Sigma_b$.
Since $\cG$ is supported on $\Sigma_b$, we have that $u_b^*\cG\cong u_b^!\cG$ (e.g., see \cite[Cor.8.2.10]{Max}). Hence \[ \cG \cong {u_b}_! u_b^!\cG \cong {u_b}_* u_b^* \cG.\]
Using proper base change and the above diagram, we have:
\be\label{bcs}
\begin{split}
\bH^k(F_b^\infty; i_b^! \cG) &\cong \bH^k(F_b^\infty; i_b^! {u_b}_* u_b^* \cG) \cong \bH^k(F_b^\infty;  {u^\infty_b}_* (\sigma_b^\infty)^!u_b^! \cG) 
\cong \bH^k(\Sigma_b^\infty;  (\sigma_b^\infty)^!u_b^! \cG) \\
&\cong  \bH^k(\Sigma_b^\infty;  (u_b^\infty)^!i_b^! \cG).
\end{split}
\ee
Similarly, one obtains:
\[
\bH^k(\oF_b; R{j_b}_* j_b^* \cG)\cong \bH^k(F_b; j_b^* \cG) \cong \bH^k(\Sigma^a_b; (u^a_b)^*j_b^* \cG).
\] $\hfill\qed$
\er

A useful consequence of the proof of the above Theorem \ref{tat} is the following.
\bp\label{pr43}
With the above notations, we have
\be\label{cons} \wti{H}^{n-s-1}(F_b;\bQ) \hookrightarrow \bH^{n-s}(F_b^\infty; i_b^! \cF) \cong \ker\left(\bH^{-s}(F_b^\infty; i_b^! \cG) \overset{M_b-id}{\lra} \bH^{-s}(F_b^\infty; i_b^! \cG)\right),
\ee
with $\cF:=\ok_b^*Rj_*\bQ_X$ and  $\cG={^p\varphi}_{\of-b}Rj_*\bQ_X[n+1]$.
\ep 
\begin{proof} The  inclusion in \eqref{cons} follows from \eqref{tvan}, \eqref{want}, the long exact sequence for the reduced cohomology of the pair $(T(F_b), F_b)$, together with \eqref{e42}. The isomorphism in \eqref{cons} follows from \eqref{s47} and \eqref{e44}, where we also denote by $M_b$ the automorphism induced on $\bH^{-s}(F_b^\infty; i_b^! \cG)$ from the monodromy of $\cG$. 
\end{proof}

\br
If $s=0$ and $\Sigma_b^\infty:=\Sigma_b \cap \oX_\infty\neq \emptyset$, then \eqref{cons} yields the upper bound on $\dim \wti{H}^{n-1}(F_b;\bQ)$ obtained (by the exact same method) in \cite[Thm.6.3.23]{Di}. We aim to generalize this fact below, to obtain a bound on $\dim \wti{H}^{n-s-1}(F_b;\bQ)$ (i.e., the first possibly non-vanishing Betti number of $F_b$) for an arbitrary non-negative integer $s$. $\hfill\qed$
\er

\br[Tame polynomials] If $f\colon \bC^{n+1}\to \bC$ is a tame polynomial then, as seen in Remark \ref{rem33}, we have $s=0$ and $\Sigma_b^\infty= \emptyset$.
In view of \eqref{bcs}, this further implies via \eqref{s47} and \eqref{e42} that $H^k(T(F_b),F_b;\bQ)=0$ for all integers $k$. The cohomology long exact sequence of the pair then implies that one has an isomorphism
\[ H^k(T(F_b);\bQ) \cong H^k(F_b;\bQ) \]
for every integer $k$; compare also with \cite[Thm.6.3.17(iii)]{Di}.  $\hfill\qed$
\er

Without any loss of generality, we can assume that $\Sigma_b^\infty$ is a union of strata of $\Sigma_b$. Then the following holds.
\bt
\label{ibinc}
Assume that the complex polynomial $f\colon \bC^{n+1} \to \bC$ has $s$-dimensional singularities, including at infinity. Let $b \in B$ be a bifurcation value of $f$, with corresponding atypical fiber $F_b=f^{-1}(b)$. 
With the notations used in this section, let us further denote by $\delta_s:\Sigma_b^{\infty,s} \hookrightarrow \Sigma_b^\infty$ the inclusion of the union of (top) $s$-dimensional strata contained in $\Sigma_b^\infty$. Then we have an $M_b$-equivariant inclusion
\be\label{inclb}
\bH^{-s}(F_b^\infty; i_b^! \cG) \hookrightarrow \bH^{-s}(\Sigma_b^{\infty,s}; \delta_s^! (u_b^\infty)^!i_b^! \cG) \cong H^0(\Sigma_b^{\infty,s}; \cH^{-s}(\delta_s^! (u_b^\infty)^!i_b^! \cG)).
\ee
\et

\begin{proof}
Recall here the relevant inclusions
\[
\xymatrix{
 \oF_b  & \ar[l]_{i_b} F_b^\infty & \\
 \Sigma_b \ar[u]^{u_b} & \ar[u]_{u^\infty_b} \ar[l]^{\sigma_b^\infty} \Sigma_b^\infty  & \ar[l]^{\delta_s} \Sigma_b^{\infty,s}.
}
\]
%\[ \Sigma_b^{\infty,s} \overset{\delta}{\hookrightarrow} \Sigma_b^\infty \overset{u_b^\infty}{\hookrightarrow} F_b^\infty \overset{i_b}{\hookrightarrow} \oF_b \]

If $s=0$, the statement follows from \eqref{bcs}, since $\delta_s$ is in this case the identity map, and \eqref{inclb} is in fact an isomorphism.

Let us now assume that $s\geq 1$, and recall from \eqref{bcs} that 
$$
\bH^{-s}(F_b^\infty; i_b^! \cG) 
 \cong  \bH^{-s}(\Sigma_b^\infty;  (u_b^\infty)^!i_b^! \cG)
 \cong \bH^{-s}(\Sigma_b^\infty;  (\sigma_b^\infty)^!u_b^! \cG),
$$
where in the notations of \eqref{pp} we have that 
\[ u_b^! \cG \cong u_b^* \cG =\cP_b,\]
which is a perverse sheaf on $\Sigma_b$. The proof of the inclusion part of \eqref{inclb}  follows now as in the proof of Theorem \ref{tat}, using attaching triangles for strata in $\Sigma_b^\infty$, the hypercohomology spectral sequence, together with the costalk condition for the perverse sheaf $\cP_b$ applied to strata contained in $\Sigma_b^\infty$ (regarded as strata in $\Sigma_b$). The isomorphism part of \eqref{inclb} follows by using the hypercohomology spectral sequence associated to $\bH^{-s}(\Sigma_b^{\infty,s}; \delta_s^! (u_b^\infty)^!i_b^! \cG)$, using the costalk condition for $\cP_b$ applied to $s$-dimensional strata.  We leave the details to the reader.
\end{proof}

As a consequence of Theorems \ref{tat} and \ref{ibinc}, we have the following.
\bc\label{cor43}
If, under the assumptions of Theorem \ref{tat}, we assume, moreover, that 
\[\dim(\Sigma_b^\infty) <s,\] 
e.g., if $s_b=\dim \Sigma_b <s$, then 
\[  \wti{H}^{n-s-1}(F_b;\bQ)=0.\]
\ec

\begin{proof}
%If $s=0$, the assertion follows immediately from \eqref{cons} in view of \eqref{bcs}, since $\Sigma_b^\infty$ is in this case the empty set. If $s>0$, t
The  assertion follows immediately from \eqref{cons} and \eqref{inclb}, since $\Sigma_b^{\infty,s}=\emptyset$ by assumption.
\end{proof}

Another consequence of Theorem \ref{ibinc} is obtaining an upper bound on the Betti number $\dim  \wti{H}^{n-s-1}(F_b;\bQ)$. Before stating the result, let us introduce some notation. 

For an $s$-dimensional connected stratum $S$ in $\Sigma_b^{\infty}$ (and hence also a top stratum in $\Sigma_b$), let $MF^\pitchfork_{S}$ denote the (transversal) Milnor fiber of $\of-b$ at some point in $S$. Its reduced cohomology is concentrated in degree $n-s$. Moreover, the proof of Lemma \ref{identif3} together with \eqref{mfi} show that $\wti{H}^{n-s}(MF^\pitchfork_{S};\bQ)^\vee$ is isomorphic to the stalk $L_S$ of the local system $\cL_S:=\cH^{-s}(\cP_b)\vert_S$ on $S$ with corresponding local system monodromy
$A_S:\pi_1(S) \to {\rm Aut}\left(L_S\right)$.
In these notations, we have the following.
\bt\label{topbb}
Assume that the complex polynomial $f\colon \bC^{n+1} \to \bC$ has $s$-dimensional singularities, including at infinity. Let $b \in B$ be a bifurcation value of $f$, with corresponding atypical fiber $F_b=f^{-1}(b)$, and assume that $\dim(\Sigma_b^\infty) =s$. Then
\be\label{bbound}
\dim  \wti{H}^{n-s-1}(F_b;\bQ) \leq \sum_{S \subset \Sigma_b^\infty \atop \dim S=s} \dim \ \ker \left(M_b - id:  (L_S)^{A_S} \to (L_S)^{A_S} \right),
\ee
where the sum on the right-hand side of the inequality is over $s$-dimensional Whitney strata contained in $\Sigma_b^\infty$, and $M_b$ is induced from the corresponding monodromy automorphism of ${^p\varphi}_{\of-b}Rj_*\bQ_X[n+1]$. In particular, 
\be\label{bbound2}
\dim  \wti{H}^{n-s-1}(F_b;\bQ) \leq \sum_{S \subset \Sigma_b^\infty \atop \dim S=s} \mu^\pitchfork_S,
\ee
with $\mu^\pitchfork_S:=\dim \wti{H}^{n-s}(MF^\pitchfork_{S};\bQ)$ the corresponding transversal Milnor number of $\of-b$ along the stratum $S \subset \Sigma_b^\infty$.
\et

\begin{proof}
In view of Proposition \ref{pr43} and Theorem \ref{ibinc}, and in the notations of the latter, it suffices to study the group
\[H^0(\Sigma_b^{\infty,s}; \cH^{-s}(\delta_s^! (u_b^\infty)^!i_b^! \cG)) \cong 
\bigoplus_{S \subset \Sigma_b^{\infty,s}} H^0(S; \cH^{-s}(\delta_s^! (u_b^\infty)^!i_b^! \cG)\vert_S).  \]
We have,
\[\delta_s^! (u_b^\infty)^!i_b^! \cG = \delta_s^! (\sigma_b^\infty)^!\cP_b=  (\sigma_b^\infty \circ \delta_s)^!\cP_b = (\sigma_b^\infty \circ \delta_s)^*\cP_b,\]
where $\cP_b$ is the perverse sheaf on $\Sigma_b$ defined as in \eqref{pp}, and we use the fact that \[\sigma_b^\infty \circ \delta_s:\Sigma_b^{\infty,s} \hookrightarrow \Sigma_b\] is the open inclusion in $\Sigma_b$ of the top $s$-dimensional strata ``at infinity''. Hence,
\[\cH^{-s}(\delta_s^! (u_b^\infty)^!i_b^! \cG)\vert_S \cong \cH^{-s}( (\sigma_b^\infty \circ \delta_s)^*\cP_b)\vert_S \cong  \cH^{-s}(\cP_b)\vert_S\]
is the above mentioned local system $\cL_S$ on $S$, with stalk $L_S\cong \wti{H}^{n-s}(MF^\pitchfork_{S};\bQ)^\vee$ and monodromy $A_S$. It follows that
\[H^0(S; \cH^{-s}(\delta_s^! (u_b^\infty)^!i_b^! \cG)\vert_S) \cong (L_S)^{A_S}.\]
 The desired inequality \eqref{bbound} is a consequence of \eqref{cons}, \eqref{inclb} and the above considerations.
\end{proof}

As a consequence of Theorems \ref{tat} and \ref{topbb}, we get the following result discussed in \cite[Thm.6.3.23]{Di}.

\bc\label{topbbc}
Assume that the complex polynomial $f\colon \bC^{n+1} \to \bC$ has only isolated singularities, including at infinity. Let $b \in B$ be a bifurcation value of $f$, with corresponding atypical fiber $F_b=f^{-1}(b)$. Then
\begin{center}
$\wti{H}^k(F_b;\bQ)=0$ \ \ for all \ $k < n-1$
\end{center}
and
\be\label{bboundc}
\dim  \wti{H}^{n-1}(F_b;\bQ) \leq \sum_{p \in \Sigma_b^\infty} \mu_p,
\ee
with $\mu_p:=\dim \wti{H}^{n}(MF_p;\bQ)$ the corresponding Milnor number of $\of-b$ at the singular point $p\in \Sigma_b^\infty$.
\ec

\begin{proof}
When $s=0$, the strata $S$ appearing in the sum of \eqref{bbound} are points $p \in \Sigma_b^\infty$, and the local system monodromies $A_S$ are all trivial. Moreover, $\dim \ker(M_b-id)\leq \dim L_p = \mu^\pitchfork_p=\mu_p$, hence \eqref{bbound} yields \eqref{bboundc}.
\end{proof}

%%%%%%%%%%%%%%%%%%

\section{Top Betti number of general fiber. Semi-continuity properties}\label{top}
Recall from \cite{T2} that a complex polynomial $f\colon\C^{n+1} \to \C$ is called {\it generic-at-infinity} (or, of $\mathcal{G}$-type) if, for any $c\in \C$, the compactified fiber $\oF_c \subset \C P^{n+1}$ is transversal to the hyperplane at infinity $H_\infty$. It can be shown that fibers of a $\cG$-type polynomial may have at most isolated singularities. %One important example is that of the Fermat polynomial $f(x_1, \ldots, x_{n+1})=\sum_{i=1}^{n+1} x_i^d.$

Degree $d$ polynomials of $\cG$-type are completely characterized by the fact that the top Betti number $b_n$ of their general fiber equals $(d-1)^{n+1}$, e.g., see  \cite[Prop.4.1.5]{T2}. The lower semi-continuity property for the top Betti number of general fibers (cf. \cite[Prop.2.1]{ST1}, \cite[Prop.4.2.1]{T2}) then yields the following upper bound result.

\begin{proposition}\label{pr61}
Let $f:\C^{n+1} \to \C$ be a degree $d$ complex polynomial with bifurcation set $B$. For any $c \in \C \setminus B$, one has
\be\label{tb}
b_n(F_c)\leq (d-1)^{n+1}.
\ee
\end{proposition}

In this section, we give a sheaf theoretic proof of the above-mentioned lower semicontinuity property for the top Betti number of general fibers, which has the advantage that it is also amenable to Hodge-theoretic considerations. (The homological version of this result is proved in \cite{ST1,T2} by homotopy-theoretic methods.)

\begin{theorem}\label{lsc}
%Let $f,g: \C^{n+1} \to \C$ be degree $d$ polynomials, with $P(x,s):=f(x)-sg(x):\C^{n+1} \times \C \to \C$ a one-parameter deformation of $f=P(-,0)$ by $g$. 
Let $P(x,s):\C^{n+1} \times \C \to \C$ be a one-parameter family of degree $d$ polynomials, %with $f=f_0=P(-,0)$.
and denote by $F^s$ the general fiber of $f_s:=P(-,s)$, $s \in \C$. Then there is an injective homomorphism 
\be\label{injl} H^n_c(F^0;\Q)  \rightarrowtail H^n_c(F^s;\Q) \ee
for $s\neq 0$ close enough to $0$. In particular, the top Betti number $b_n$ of general fibers satisfies the lower semi-continuity property, namely,
\begin{center} $b_n(F^s) \geq b_n(F^0)$, for $s \neq 0$ close enough to $0$.\end{center} 
\end{theorem}

\begin{proof}
For $c \in \C$ general enough, we may assume that $f_s^{-1}(c)=F^s$ is the general fiber of $f_s$, for $s$ in a small enough neighborhood of $0$. Consider the variety
\[
\overline{Y}:=\{\left([x_0:x],s\right) \in  \bC P^{n+1} \times \bC \mid P^h(x_0,x,s)-cx_0^d=0  \},
\]
where $P^h$ is the homogenization of $P$ by the variable $x_0$, considering $s$ as a parameter. Let $\pi:\overline{Y} \to \C$ be the projection to the $s$-coordinate. Note that $\oF^s:=\pi^{-1}(s)$ is the projective closure in $\C P^{n+1}$ of the affine hypersurface $F^s=f_s^{-1}(c)$.

Set \[ Y = \overline{Y}\cap \{x_0=1\}=\{(x,s) \in  \bC^{n+1} \times \bC \mid P(x,s)=c  \},\] with the open affine inclusion $u: Y\hookrightarrow \overline{Y}$, and note that $F^s=\oF^s \cap Y$.

Let $\overline{i}_0: \overline{F}^0 \hookrightarrow \overline{Y}$ be the inclusion map, and consider the distinguished triangle 
\[ \overline{i}_0^{*}u_{!}\Q_Y \lra \psi_{\pi}u_{!}\Q_Y \lra \varphi_{\pi}u_{!}\Q_Y \overset{[1]}{\lra}\]
and the following part of the associated hypercohomology long exact sequence:
\be\label{les10} \cdots \to \bH^{n-1}(\overline{F}^0; \varphi_{\pi}u_{!}\Q_Y) \to \bH^n(\overline{F}^0; \overline{i}_0^{*}u_{!}\Q_Y) \to \bH^{n}(\overline{F}^0; \psi_{\pi}u_{!}\Q_Y) \to \cdots \ee

Let $i_0:F^0 \hookrightarrow Y$ and $u_0:F^0 \hookrightarrow \overline{F}^0$ be the inclusion maps. 
Using proper base change, we have
\begin{equation}\label{les11}
\bH^n(\overline{F}^0; \overline{i}_0^{*}u_{!}\Q_Y) \cong 
   \mathbb{H}^n(\overline{F}^0; {u_0}_!{i}_0^{*}\mathbb{Q}_{Y}) \cong  \mathbb{H}^n(\overline{F}^0; {u_0}_!\mathbb{Q}_{F^0})  \cong H_c^n(F^0;\mathbb{Q}).
\end{equation}

Similarly, let $\overline{i}_{s}: \overline{F}^{s}\hookrightarrow \overline{Y}$, $u_{s}: F^{s} \hookrightarrow \overline{F}^{s}$ and $i_{s}: {F}^{s}\hookrightarrow Y$ be the inclusions. Then, using \cite[Ex.10.4.20]{MS} and proper base change, for $s\neq 0$ sufficiently close to $0$ we have
\begin{equation}\label{les12}
\bH^{n}(\overline{F}^0; \psi_{\pi}u_{!}\Q_Y) \cong 
 \mathbb{H}^n(\overline{F}^{s}; \overline{i}_{s}^{*}u_{!}\mathbb{Q}_{Y})
 \cong \mathbb{H}^n(\overline{F}^{s}; {u_{s}}_!i_{s}^{*}\mathbb{Q}_{Y}) \cong H^n_c(F^{s};\mathbb{Q}).
\end{equation}

Therefore, by combining \eqref{les10}, \eqref{les11} and \eqref{les12}, we note that in order to get an injection $H^n_c(F^0;\Q)  \rightarrowtail H^n_c(F^s;\Q)$ as in the statement of the theorem, it suffices to show that  $\bH^{n-1}(\overline{F}^0; \varphi_{\pi}u_{!}\Q_Y) \cong 0$.

Note that
\begin{equation}\label{lab1}
\bH^{n-1}(\overline{F}^0; \varphi_{\pi}u_{!}\Q_Y) \cong 
\mathbb{H}^{-1}(\overline{F}^0;{}^p\varphi_{\pi}u_{!}\mathbb{Q}_{Y}[n+1])
\cong H^{-1}(R{\pi_0}_*{}^p\varphi_{\pi}u_{!}\mathbb{Q}_{Y}[n+1]),
\end{equation}
where ${}^p\varphi_{\pi}=\varphi_{\pi}[-1]$ is the perverse vanishing cycle functor,  and $\pi_0\colon \overline{F}^0\rightarrow \{0\}$ is the constant map.
 Now, consider the maps $\overline{Y}\xrightarrow{\pi}\mathbb{C}\xrightarrow{\textrm{id}}\mathbb{C},$ with induced maps $\overline{F}^0\xrightarrow{\pi_0} \{0\}\xrightarrow{\textrm{id}} \{0\}$ of zero sets. Since $\pi$ is proper, by proper base change for vanishing cycles (e.g.,  see \cite[Rem.4.3.7]{Sch}) we have
\begin{equation}
    \varphi_{\textrm{id}}\circ R\pi_{*}\simeq R{\pi_{0}}_*\circ \varphi_{\pi}
\end{equation}
and so
\begin{equation}\label{lab2}
H^{-1}(R{\pi_0}_*{}^p\varphi_{\pi}u_{!}\mathbb{Q}_{Y}[n+1])\cong H^{-1}({}^p\varphi_{\textrm{id}} R\pi_{*}u_{!}\mathbb{Q}_{Y}[n+1]) \cong H^{-1}({}^p\varphi_{\textrm{id}}R(\pi \circ u)_{!}\mathbb{Q}_Y[n+1]),
\end{equation}
where the last isomorphism uses the properness of $\pi$. 

Since $Y$ is a complex hypersurface of dimension $n+1,$ the complex $\mathbb{Q}_Y[n+1]$ is perverse on $Y$. Since $\pi \circ u : Y \to \C$ 
is an affine map, the functor $R(\pi \circ u)_{!}$ is left $t$-exact, and so $R(\pi \circ u)_{!}\mathbb{Q}_Y[n+1]\in {}^p D^{\geq 0}(\mathbb{C}).$ Then since ${}^p\varphi_{\textrm{id}}$ is $t$-exact, we get that ${}^p\varphi_{\textrm{id}} R(\pi \circ u)_{!}\mathbb{Q}_Y[n+1]\in {}^pD^{\geq 0}(\{0\})=D^{\geq 0}(\{0\}),$ which yields that 
\begin{equation}\label{lab3}
    H^{-1}({}^p\varphi_{\textrm{id}}R(\pi \circ u)_{!}\mathbb{Q}_Y[n+1])=0.
\end{equation}
Then by \eqref{lab1},  \eqref{lab2} and \eqref{lab3} we get that
\begin{equation}
    \bH^{n-1}(\overline{F}^0; \varphi_{\pi}u_{!}\Q_Y) \cong 0,
\end{equation}
as desired.

Finally, since both $F_0$ and $F_{s}$ are smooth of complex dimension $n$, we get from \eqref{injl} by Poincar\'e duality that \[b_n(F^0) \leq b_n(F^{s}).\]
\end{proof}

\begin{proof}[Proof of Proposition \ref{pr61}]
This is a direct consequence of the above theorem, after choosing a one-parameter degree $d$ polynomial deformation $P(x,s)$ of $f=P(-,0)$, so that $P(x,s)$ is generic-at-infinity for some $s\neq 0$ close enough to $0$.
\end{proof}

\section{Examples}\label{ex}
Examples are generally difficult to handle. While the case of isolated singularities, including those at infinity, is fairly well understood, in this section we present several examples of complex polynomials with non-isolated singularities and apply the bounds from Corollary~\ref{nmsb} and Theorem~\ref{topbb} to these cases.

\bex\label{ex61}
Consider the following extension of Broughton's example:
\[ f: \C^{n+1} \to \C, \ f(x_1, \ldots, x_{n+1})= x_1^2x_2+x_1 \ \ (n \geq 2).\]
Then a simple calculation shows that $f$ has no affine singularities, however, $c=0$ is an atypical value of $f$. In fact, if $c\neq 0$, $F_c=f^{-1}(c)$ is isomorphic to $\C^* \times \C^{n-1}$, while $F_0=f^{-1}(0)$ consists of two disjoint irreducible components, $\C^n$ and $\C^*\times \C^{n-1}$. Therefore, 
\[\wti{H}^k(F;\Q) \cong \begin{cases}
\Q & k=1 \\
0 & k \neq 1,
\end{cases}\]
\[\wti{H}^k(F_0;\Q) \cong \begin{cases}
\Q & k=0,1 \\
0 & {\rm otherwise}.
\end{cases}\]

\noindent One also gets by direct calculation that
\[
 \Sigma_0=\Sigma_0^\infty
 =\{[x_0:x_1:x_2:\cdots:x_{n+1}]\mid x_0=x_1=0\}
 \cong{\bC P}^{n-1},
\]
so $s=s_0=n-1$. A compatible Whitney stratification can be chosen with
unique top-dimensional stratum
\[
 S_{n-1}:=\Sigma_0^\infty\cap\{x_2\ne0\}
 \cong\mathbb C^{n-1},
\]
and lower-dimensional strata contained in
$\Sigma_0^\infty\cap\{x_2=0\}\cong{\bC P}^{n-2}$.

To compute the transversal Milnor(-L\^e) number $\mu^\pitchfork$ along $S_{n-1}$, note that
\[
 \oX=\{x_1^2x_2+x_1x_0^2-tx_0^3=0\}
 \subset\mathbb{CP}^{n+1}\times\mathbb C.
\]
At a point $p\in S_{n-1}$, use the chart $x_2=1$ and fix the remaining
coordinates $x_3,\ldots,x_{n+1}$. The transverse graph and its projection
are
\[
 \overline Y=\{h_t(x_0,x_1)=x_1^2+x_1x_0^2-tx_0^3=0\}
\]
and $\overline g=t$. 
Let $\overline G_t=\bar g^{-1}(t)$. At $t=0$, completing the square gives
\[
 h_0=\left(x_1+\frac{x_0^2}{2}\right)^2-\frac{x_0^4}{4},
\]
so $(\overline G_0,p)$ is a plane curve singularity of type $A_3$ and
$\mu(\overline G_0,p)=3$. For $t=c\ne0$ sufficiently small,
\[
 h_c=\left(x_1+\frac{x_0^2}{2}\right)^2
       -x_0^3\left(c+\frac{x_0}{4}\right).
\]
Since the factor $c+\frac{x_0}{4}$ is a unit at the origin, $(\overline G_c,p)$ is an $A_2$ singularity, with $\mu(\overline G_c,p)=2$. There are no other singular points of the nearby transverse curve in a sufficiently small Milnor neighborhood. Therefore, the Milnor-number jump formula at infinity \eqref{MilLe}, discussed in Remark \ref{calcmu}, yields
\[
 \mu^\pitchfork
 =\dim\widetilde H^1(MF^\pitchfork;\mathbb Q)
 =\mu(\overline G_0,p)-\mu(\overline G_c,p)=3-2=1.
\]
Here $MF^\pitchfork$ is the transversal Milnor fiber of $\of$, equivalently the Milnor fiber of the graph projection $\overline g$ at $p$. Thus Corollary~\ref{nmsb} and Theorem~\ref{topbb} give, respectively,
\[
 b_1(F)\le1,\qquad \widetilde b_0(F_0)\le1.
\]
Both bounds are sharp, since the concrete descriptions of the fibers give $b_1(F)=\widetilde b_0(F_0)=1$.
\eex

\bex
Let $f:\bC^4\to \bC$ be given by $f=x_1^2+x_1(x_2^2+x_3^2+x_4^2)$. This is a weighted homogeneous polynomial, so $B=\{0\}$ and $F_0=f^{-1}(0)$ is contractible. The first non-vanishing cohomology of the general (i.e., Milnor) fiber $F$ of $f$ can already be computed as in \cite{MPT2}, i.e., all top-dimensional singular strata contributing to this cohomology arise from the affine part (cf. also Remark \ref{whp3}). In more detail, the singular locus of $f$ consists of 
$\Sigma=V(x_1,x_2^2+x_3^2+x_4^2),$ with Whitney strata $S_0=\{(0,0,0,0)\}$ and $S_2=\Sigma \setminus S_0$. 
The transversal Milnor fiber $MF^\pitchfork$ for the stratum $S_2$ is the Milnor fibre of an $A_{1}$-singularity, so $\mu^\pitchfork=1$. Moreover,  $S_2$ is homotopy equivalent to the link of a quotient surface singularity of type $A_1$, thus $\pi_1(S_2)=\bZ/2$. 
Therefore, using \eqref{lastf}, we get that
$$H^1(F;\Q) \hookrightarrow  H^1(MF^\pitchfork;\Q)^{\bZ/2} \subset \Q,$$
whence $b_1(F)\leq 1$.
Note that after a change of coordinates $F$ can be identified with the Milnor fiber of the singularity at the origin of the polynomial  $x_1^2+(x_2^2+x_3^2+x_4^2)^2$, and its homotopy type can be deduced via the Thom-Sebastiani theorem as the suspension on two disjoint $S^2$'s. That is, $F$ has the homotopy type of $S^1\vee S^3 \vee S^3$. This implies that in fact $b_1(F)=1$, and hence $\pi_1(S_2)$ acts trivially on $H^1(MF^\pitchfork;\Q)$. 
\eex

\bex
Let $f:\bC^3\to \bC$ be given by $f=x_1+x_1^2x_2x_3$. Then $\Sing(f)=\emptyset$, so all fibers of $f$ are smooth. 

It is easy to see that the zero-fiber of $f$, that is, $F_0=\{x_1+x_1^2x_2x_3=0\}$ is a disjoint union  of a copy of $\mathbb{C}^2$ (where $x_1=0$) with a $2$-dimensional complex torus $(\C^*)^2$. Hence its Betti numbers are $b_0(F_0) = 2$, $b_1(F_0) = 2$, $b_2(F_0) = 1$, and $b_i(F_0) = 0$, for all $i \geq 3$.

Let us now consider the fiber $F_1=\{x_1+x_1^2x_2x_3=1\}\subset \mathbb{C}^3.$ Consider the map $u: \mathbb{C}^3\rightarrow\mathbb{C}^3$ given by $(x_1,x_2,x_3)\mapsto (x_1,y:=x_1^2x_2,x_3).$ Note that if $(x_1,x_2,x_3)\in F_1,$ then $u(x_1,x_2,x_3)$ lies on the hypersurface $Y_1=\{x_1+yx_3=1\}\subset \mathbb{C}^3.$ Let $g: F_1\rightarrow Y_1$ denote the restriction of $u$ to $F_1.$ Since $x_1\neq 0$ in $F_1,$ the image of $g$ in $Y_1$ also satisfies $x_1\neq 0.$ The  map $h: Y_1 \setminus \{x_1=0\}\rightarrow F_1$ given by $(x_1,y,x_3)\mapsto (x_1, x_2=\frac{y}{x_1^2}, x_3)$ shows that $g$ is a homeomorphism $F_1\cong Y_1\setminus \{x_1=0\}.$

We will next show that one has a homeomorphism $Y_1\setminus \{x_1=0\}\cong \mathbb{C}^2\setminus \{yx_3=1\}.$ Take the maps $p: \mathbb{C}^2\setminus \{yx_3=1\}\rightarrow Y_1\setminus \{x_1=0\}$ given by $(y,x_3)\mapsto (1-yx_3, y, x_3),$ and $q: Y_1\setminus \{x_1=0\}\rightarrow\mathbb{C}^2 \setminus \{yx_3=1\}$ given by $(x_1,y,x_3)\mapsto (y,x_3).$ These two maps are inverses, proving the assertion.

So to calculate the Betti numbers of $F_1$,  it suffices to calculate the Betti numbers of $\mathbb{C}^2\setminus \{yx_3=1\}.$ From the long exact sequence of the pair $(\mathbb{C}^2,\mathbb{C}^2\setminus \{yx_3=1\}),$ it follows that for $k\geq 1,$ 
\[H^k(\mathbb{C}^2\setminus \{yx_3=1\};\bQ)\cong H^{k+1}(\mathbb{C}^2, \mathbb{C}^2\setminus \{yx_3=1\};\bQ).\]
By Alexander duality and Poincar\'e duality (noting that the curve $yx_3=1$ is smooth in $\bC^2$), we get 
\[H^{k+1}(\mathbb{C}^2, \mathbb{C}^2\setminus \{yx_3=1\};\bQ)\cong H_c^{3-k}(\{yx_3=1\};\bQ)^{\vee}\cong H_{k-1}(\{yx_3=1\};\bQ)^\vee.\]
Hence for  $k\geq 1$ we have 
\[H^k(F_1;\bQ)\cong H^k(\mathbb{C}^2\setminus \{yx_3=1\};\bQ) \cong H_{k-1}(\{yx_3=1\};\bQ)^\vee \cong H_{k-1}(\mathbb{C}^{*};\bQ)^\vee,\]
which implies \[H^2(F_1;\bQ) \cong \bQ, \ H^1(F_1;\bQ) \cong \bQ.\]
We also have $H^0(F_1;\bQ) \cong \bQ$ since $\mathbb{C}^2\setminus \{yx_3=1\}$ is connected. 
Therefore, the Betti numbers of $F_1$ are $b_k(F_1)=1$ for $k\in \{0,1,2\}$ and  they vanish otherwise.

On the other hand, for any $c\neq 0,$ the map $v_c: \mathbb{C}^3\rightarrow\mathbb{C}^3$, $(x_1,x_2,x_3)\mapsto (c^{-1}x_1, cx_2, x_3)$ is a homeomorphism, and the image of the fiber $F_c=\{x_1^2x_2x_3+x_1=c\}$ is $F_1$. Indeed, with $x=(x_1,x_2, x_3)$, we have $f(v_c(x))=f(x)/c$, so if $x \in F_c$, then $v_c(x) \in F_1$. The inverse of $v_c$ is $(x_1, x_2, x_3) \mapsto (cx_1, c^{-1}x_2, x_3)$ and this maps $F_1$ to $F_c$.
So the general fiber of $f$ has the topological type of $F_1,$ and the other possibility for the topological type of a fiber of $f$ is $F_0=\{x_1^2x_2x_3+x_1=0\}.$ The zero-fiber has already been shown to be disconnected, while $F_1$ is connected.
So $0$ is the only bifurcation value of $f$.

The compactification $\overline{F}_0=\{x_1^2x_2x_3+x_1x_0^3=0\}\subset \mathbb{C}P^3$ has two irreducible components, the hyperplane $H=\{x_1=0\}$ and the cubic hypersurface $C=\{x_0^3+x_1x_2x_3=0\}$. Their intersection is
\[ L=H \cap C=\{x_0=0, x_1=0\}=\{[0:0:x_2:x_3]\}\cong \C P^1,\]
a projective line contained in the hyperplane at infinity $\{x_0=0\}$. The component $C$ has $3$ singular points,
\[P_1=[0:1:0:0], \ P_2=[0:0:1:0], \ P_3=[0:0:0:1], \]
and we note that $P_1\in C \setminus H$, while $P_2, P_3 \in L$. 
So the singular locus of $\oF_0$ consists of $L\sqcup P_1$.  
A natural Whitney stratification of $\oF_0$ has only one $1$-dimensional stratum
\[S_1=L\setminus \{P_2, P_3\}=\{[0:0:x_2,x_3] \mid x_2 \neq 0, x_3\neq 0\},\]
which is contained entirely in the hyperplane at infinity. 

To find the transversal Milnor number $\mu^\pitchfork$ along $S_1$, consider the graph compactification
\[
 \oX=\{x_1^2x_2x_3+x_1x_0^3-tx_0^4=0\}
 \subset\bC P^3\times\mathbb C.
\]
At a point $p\in S_1$, work in the chart $x_2=1$ and take the transverse section $x_3=a$, where $a\ne0$. In the notation of Lemma~\ref{identif3}, the transverse graph is
\[ \overline Y=\{h_t(x_0,x_1)=a x_1^2+x_1x_0^3-tx_0^4=0\},\]
with projection $\overline g=t$. Denote by $\overline G_t=\overline g^{-1}(t)$ the fiber of $\overline g$ over $t \in \bC$. Then 
\[
 \overline G_0=\{x_1(x_0^3+a x_1)=0\}.
\]
We can write
\[
 h_0=a\left(x_1+\frac{x_0^3}{2a}\right)^2
           -\frac{x_0^6}{4a},
\]
so $(\overline G_0,p)$ is an $A_5$ singularity with $\mu(\overline G_0,p)=5$. For $t=c\ne0$ sufficiently small,
\[
 h_c=a\left(x_1+\frac{x_0^3}{2a}\right)^2
           -x_0^4\left(c+\frac{x_0^2}{4a}\right).
\]
The factor $c+\frac{x_0^2}{4a}$ is a unit at the origin. Hence the nearby transverse curve has an $A_3$ singularity at $p$, with $\mu(\overline G_c,p)=3$, and no other singular points in a sufficiently small Milnor neighborhood.
Therefore, the Milnor-number jump formula \eqref{MilLe}, discussed in Remark \ref{calcmu}, yields
\[
 \mu^\pitchfork
 =\dim\widetilde H^1(MF^\pitchfork;\mathbb Q)
 =\mu(\overline G_0,p)-\mu(\overline G_c,p)=5-3=2.
\]
Here $MF^\pitchfork$ is the transversal Milnor fiber of $\of$ along $S_1$. Since $S_1$ is the unique connected one-dimensional stratum contributing to these bounds, Corollary~\ref{nmsb} gives $b_1(F)\le2$, while Theorem~~\ref{topbb}  gives $\widetilde b_0(F_0)\le2$. The direct computations above show that $b_1(F)=\widetilde b_0(F_0)=1$, so these bounds are not equalities.
\eex

\end{document}